\documentclass[12pt]{amsart}
\usepackage[alphabetic]{amsrefs}
\usepackage{mathdots, blkarray}
\usepackage{multicol}
\usepackage{graphicx}
\usepackage{caption}
\usepackage{amscd}
\usepackage{upgreek}
\usepackage{stmaryrd}
\usepackage{longtable}
\usepackage[T1]{fontenc}
\usepackage{latexsym, amsmath, amssymb, amsthm}
\usepackage[svgnames]{xcolor}
\usepackage{rsfso}
\usepackage[mathscr]{eucal}
\usepackage{mathtools}
\usepackage{mathptmx}
\usepackage{titletoc}
\usepackage{wrapfig}
\usepackage{float}
\usepackage{xypic}
\usepackage{microtype}
\usepackage{dsfont}
\usepackage{xcolor}
\usepackage{color}
\usepackage{mathrsfs}
\usepackage[colorlinks = true,
            linkcolor  = DarkBlue,
            urlcolor   = DarkRed,
            citecolor  = DarkGreen]{hyperref}
\usepackage{hyperref}
\allowdisplaybreaks	
\usepackage[centering, includeheadfoot, hmargin=1.0in, tmargin=0.5in,
  bmargin=0.6in, headheight=6pt]{geometry}
\newtheorem{theorem}{Theorem}[section]
\newtheorem{corollary}{Corollary}[section]
\newtheorem{lemma}{Lemma}[section]
\newtheorem{proposition}{Proposition}[section]
\theoremstyle{definition}
\newtheorem{definition}{Definition}[section]
\theoremstyle{remark}
\newtheorem{remark}{Remark}[section]
\newtheorem{example}{Example}[section]
\numberwithin{equation}{section}

\newcommand{\bslash}{\kern-0.1em\texttt{\scalebox{0.6}[1]{/}}\kern-0.15em \texttt{\scalebox{0.6}[1]{/}}}

\begin{document}
\title[The completeness and congruences of quasi-Boolean algebras]{The completeness and congruences of quasi-Boolean algebras}

\author{Xiaohao Liu}
\address{School of Mathematical Sciences, University of Jinan, No. 336, West Road of Nan Xinzhuang,
         Jinan, Shandong, 250022 P.R. China.}
\email{liuxiaohao@stu.ujn.edu.cn}

\author{Heyan Wang}
\address{School of Mathematical Sciences, University of Jinan, No. 336, West Road of Nan Xinzhuang,
         Jinan, Shandong, 250022 P.R. China.}
\email{wangheyan@stu.ujn.edu.cn}

\author{Wenjuan Chen}
\address{School of Mathematical Sciences, University of Jinan, No. 336, West Road of Nan Xinzhuang,
         Jinan, Shandong, 250022 P.R. China.}
\email{wjchenmath@gmail.com}

\begin{abstract}
Quasi-Boolean algebras were introduced as the generalization of Boolean algebras in the setting of quantum computation logic. In this paper, we investigate the completeness and congruences of quasi-Boolean algebras. First, we discuss the number of finite quasi-Boolean algebras and characterize the finite irreducible quasi-Boolean algebras. Second, we show the standard completeness of quasi-Boolean algebras. Finally, we prove that the variety of quasi-Boolean algebras satisfies the congruence extension property and provide a complete characterization of how congruences on a Boolean subalgebra can be extended to the whole quasi-Boolean algebra.

\end{abstract}


\maketitle
\baselineskip=16.2pt

\dottedcontents{section}[1.16cm]{}{1.8em}{5pt}
\dottedcontents{subsection}[2.00cm]{}{2.7em}{5pt}
\section{Introduction}
\label{intro}

Boolean algebras, as the complemented distributive lattices, play an important role in the classical propositional logic and non-classical logic. Over the past few decades, many scholars have generalized Boolean algebras while preserving the lattice structure such as bounded De Morgan lattices \cite{DeM}, Nelson algebras \cite{Nelsonalgebras}, and others.
In 1993, Chaja introduced quasi-lattices ($q$-lattices, for short) \cite{q-lattice} and then defined an algebra of quasiordered logic as the generalization of a Boolean algebra \cite{quasiordered}.
Unlike the previously mentioned work, the algebra of quasiordered logic is based on generalizing the lattice to a $q$-lattice (or, an ordered relation to a quasi-ordered relation). In \cite{quasiordered}, Chaja showed that the quotient algebra of an algebra of quasiordered logic is a Boolean algebra, thereby establishing some relations between Boolean algebras and algebras of quasiordered logic. These works initiated the research from ordered algebras to quasi-ordered algebras.

In 2006, Ledda et al. introduced quasi-MV algebras in order to characterize the quantum computational logic \cite{qMVstandard}. The generalization from MV-algebras to quasi-MV algebras is analogous to the generalization from lattices to $q$-lattices. Ledda et al. showed that certain subalgebra of a quasi-MV algebra is an MV-algebra and this subalgebra is also isomorphic to certain quotient algebra of the quasi-MV algebra. Based on these relationships, Dvure$\breve{c}$enskij and Zahiri presented a  method for extending congruences from the MV-subalgebra to the entire quasi-MV algebra in \cite{congruenceadd}.
In 2010, Alizadeh et al. discussed the relationship between ideals (filters) of a $q$-lattice and ideals (filters) of its sublattice \cite{a11}. Such investigations can be extended to other quasi-ordered algebras. For instance, quasi-BL algebras are a type of quasi-ordered algebras that generalize both quasi-MV algebras and BL-algebras \cite{chen20}.
Chen and Xu studied the relationship between weak filters of a quasi-BL algebra and filters of its BL subalgebra \cite{chen23}, and also discussed how to get a weak filter from a filter of a BL subalgebra.

Quasi-Boolean algebras were introduced by Lv and Chen \cite{Engineering, QB} as the generalization of Boolean algebras. The original motivation for this study is to provide an algebraic framework for quantum computational logic that plays a role analogous to that of Boolean algebras in classical propositional logic.
Unlike the algebras of quasiordered logic, the unary operation of a quasi-Boolean algebra is not defined in terms of its binary operations.
Moreover, an algebra of quasiordered logic satisfying the involutive law is a quasi-Boolean algebra. Thus, quasi-Boolean algebras appear to be closer to Boolean algebras.

To characterize more universal properties and methods of quasi-ordered algebras, we further investigate quasi-Boolean algebras in this paper. On the one hand, the standard completeness of quasi-Boolean algebras should be discussed in order to prove the completeness of their corresponding logic. On the other hand, we present the relationship between quasi-Boolean algebras and their Boolean subalgebras to better understand the connections between quasi-ordered algebras and ordered algebras.
The paper is organized as follows. In Section 2, we recall some definitions and results which will be used in what follows. In Section 3, we discuss the number of finite quasi-Boolean algebras. In Section 4, we show the standard completeness theorem of quasi-Boolean algebras. Finally, we investigate the congruence extension property of the variety of quasi-Boolean algebras and provide a complete characterization of how congruences on a Boolean subalgebra can be extended to the whole quasi-Boolean algebra.
\section{Preliminary}
In this section, we recall some definitions and properties related to quasi-Boolean algebras that will be used in what follows.

\begin{definition}\label{Boolean}\cite{Boolean algebra}
An algebra $\mathbf{B}=\langle B ; \vee ,\wedge ,' ,0, 1 \rangle$ of type $\langle 2,2,1,0,0\rangle $ is called a \emph{Boolean algebra},
if it satisfies the following conditions for any $x,y,z\in B$,

(B1) $\langle B;\vee,\wedge\rangle $ is a lattice;

(B2) $x\vee (y\wedge z)=(x \vee y)\wedge (x\vee z)$ and $x\wedge (y\vee z)=(x \wedge y)\vee (x\wedge z)$;

(B3) $x\vee 1=1$ and $x\wedge 0=0$;

(B4) $x\vee x'=1$ and $x\wedge x'=0$.
\end{definition}

We abbreviate a Boolean algebra $\mathbf{B}=\langle B;\vee,\wedge,', 0,1 \rangle$ as $\mathbf{B}$ and denote the variety of all Boolean algebras by $\mathbb{B}$.

\begin{definition}\label{Definition 1}\cite{q-lattice}
An algebra $\mathbf{L}= \langle L ; \vee ,\wedge\rangle$ of type $\langle 2,2\rangle $ is called a \emph{quasi-lattice} $\mathbf{(}$q-lattice, for short$\mathbf{)}$,
if it satisfies the following conditions for any $x,y,z\in L$,

(QL1) $x\vee y=y\vee x$ and $x\wedge y=y\wedge x$;

(QL2) $x\vee (y\vee z)=(x \vee y) \vee z$ and $x\wedge (y\wedge z)=(x \wedge y)\wedge z$;

(QL3) $x\vee (x\wedge y)=x\vee x$ \,and\, $x\wedge (x\vee y)=x\wedge x$;

(QL4) $x\vee (y\vee y)=x\vee y$ \,and\, $x\wedge (y\wedge y)=x\wedge y$;

(QL5) $x\vee x=x\wedge x$.
\end{definition}

A q-lattice $\mathbf{L}=\langle L; \vee,\wedge \rangle$ is called $\emph{distributive}$, if it satisfies $ x\vee (y\wedge z)=(x\vee y)\wedge (x\vee z)$ and $ x\wedge (y\vee z)=(x\wedge y)\vee (x\wedge z)$ for any $x,y,z\in L$.

\begin{definition}\label{Definition 3}\cite{Engineering}
An algebra $\mathbf{Q} = \langle Q ; \vee ,\wedge ,^\ast, 0, 1 \rangle$ of type $\langle 2,2,1,0,0\rangle $ is called a \emph{quasi-Boolean algebra} $\mathrm{(}$QB-algebra, for short$\mathrm{)}$, if it satisfies the following conditions for any $x\in Q$,

(QB1) $\langle Q;\vee,\wedge\rangle $ is a distributive q-lattice;

(QB2) $x\vee1=1$ \,and\, $x \wedge0=0$;

(QB3) $x\vee x^{*}=1$ \,and\, $x\wedge x^{*}=0$;

(QB4) $(x\wedge x)^{*}=x^{*}\vee x^{*}$;

(QB5) $x^{**}= x$.
\end{definition}

We abbreviate a QB-algebra $\mathbf{Q} = \langle Q ; \vee ,\wedge ,^\ast, 0, 1 \rangle$ as $\mathbf{Q}$ and denote the variety of all QB-algebras by $\mathbb{QB}$.

\begin{example}\label{Example 4}\cite{QB}
The algebra $\mathbf{4} = \langle \{0,a,b,1\} ; \vee ,\wedge ,^\ast, 0, 1 \rangle$ is a QB-algebra where the operations are given by the following tables\\
$$
\begin{array}{c|cccccc} \vee & 0\; & a\; & b\; & 1\;  \\ \hline 0 & 0 & 0 & 1 & 1 \\ a & 0 & 0 & 1 & 1 \\ b& 1 & 1 & 1 & 1 \\ 1 & 1 & 1 & 1 & 1 \end{array}\qquad
\begin{array}{c|cccccc} \wedge & 0\; & a\; & b\; & 1\;  \\ \hline 0 & 0 & 0 & 0 & 0 \\ a & 0 & 0 & 0 & 0 \\ b& 0 & 0 & 1 & 1 \\ 1 & 0 & 0 & 1 & 1 \end{array}\qquad
\begin{array}{c|cccccc} x & x^{*}   \\ \hline 0 & 1 \\ a & b \\ b& a \\ 1 & 0 \end{array}
$$
\end{example}

\begin{example}\label{Example 4-}\cite{QB}
The algebra $\bar{\mathbf{4}} = \langle \{0,a,b,1\} ; \vee ,\wedge ,^\ast, 0, 1 \rangle$ is a QB-algebra where the operations are given by the following tables\\
$$
\begin{array}{c|cccc} \vee & 0\; & a\; & b\; & 1  \\ \hline 0 & 0 & 1 & 0 & 1 \\ a & 1 & 1 & 1 & 1 \\ b& 0 & 1 & 0& 1 \\ 1 & 1 & 1 & 1 & 1 \end{array}\qquad
\begin{array}{c|cccc} \wedge & 0\; & a\; & b\; & 1  \\ \hline 0 & 0 & 0 & 0 & 0 \\ a & 0 & 1 & 0 & 1 \\ b& 0 & 0 & 0 & 0 \\ 1 & 0 & 1 & 0 & 1 \end{array}\qquad
\begin{array}{c|cccc} x & x^{*}   \\ \hline 0 & 1 \\ a & b \\ b& a \\ 1 & 0 \end{array}
$$
\end{example}

\begin{example}\label{Example A}
The algebra $\mathbf{A} = \langle \{0,a,e,f,b,1\} ; \vee ,\wedge ,^\ast, 0, 1 \rangle$ is a QB-algebra where the operations are given by the following tables
$$
\begin{array}{c|cccccc}
\vee & 0   & a   & e   & f   & b   & 1   \\ \hline 0   & 0   & a   & a   & f   & f   & 1   \\a   & a   & a   & a   & 1   & 1   & 1   \\e   & a   & a   & a   & 1   & 1   & 1   \\f   & f   & 1   & 1   & f   & f   & 1   \\b   & f   & 1   & 1   & f   & f   & 1   \\1   & 1   & 1   & 1   & 1   & 1   & 1   \\
\end{array}
\qquad
\begin{array}{c|cccccc}
\wedge & 0   & a   & e   & f   & b   & 1   \\ \hline0   & 0   & 0   & 0   & 0   & 0   & 0   \\a   & 0   & a   & a   & 0   & 0   & a   \\e   & 0   & a   & a   & 0   & 0   & a   \\f   & 0   & 0   & 0   & f   & f   & f   \\b   & 0   & 0   & 0   & f   & f   & f   \\1   & 0   & a   & a   & f   & f   & 1   \\
\end{array}
\qquad
\begin{array}{c|c}
x & x^{*} \\ \hline 0 & 1   \\a & f   \\e & b   \\f & a   \\b & e   \\1 & 0   \\
\end{array}
$$
\end{example}

A QB-algebra $\mathbf{F}$ is called $\emph{flat}$, if it satisfies the equation $1=0$. We denote the variety of all flat QB-algebras by $\mathbb{FQB}$.

\begin{example}\label{Example F3}\cite{QB}
The algebra $\mathbf{F}_{3} = \langle \{0,c,d\} ; \vee ,\wedge ,^\ast, 0, 1 \rangle$ is a flat QB-algebra where the operations are given by the  following tables\\
\begin{eqnarray*}
\begin{array}{c|ccccccccc} \vee & 0\; & c& d \\ \hline 0 & 0 & 0& 0  \\ c & 0 & 0& 0 \\ d & 0 & 0& 0  \\
\end{array}\qquad
\begin{array}{c|ccccccccc} \wedge & 0\; & c & d\\ \hline 0 & 0 & 0& 0   \\ c & 0 & 0& 0 \\ d & 0 & 0& 0  \\
\end{array}\qquad
\begin{array}{c|ccccccccc} x & x^{*}   \\ \hline 0 & 0 \\ c & d\\ d & c \\  \end{array} \;\;\mathrm{and}\; 1=0.
\end{eqnarray*}
\end{example}

\begin{example}\label{Example F5}
The algebra $\mathbf{F}_{5} = \langle \{0,g,h,i,j\} ; \vee ,\wedge ,^\ast, 0, 1 \rangle$ is a flat QB-algebra where the operations are given by the  following tables\\
\begin{eqnarray*}
\begin{array}{c|ccccccccc} \vee & 0\; & g& h& i& j \\ \hline 0 & 0 & 0& 0 &0 & 0 \\ g & 0 & 0& 0 & 0& 0 \\ h & 0 & 0& 0 &0 &0  \\ i & 0 & 0& 0 & 0& 0 \\ j & 0 & 0& 0 & 0& 0 \\
\end{array}\qquad
\begin{array}{c|ccccccccc} \wedge & 0\; & g& h& i& j \\ \hline 0 & 0 & 0& 0 &0 & 0 \\ g & 0 & 0& 0 & 0& 0 \\ h & 0 & 0& 0 &0 &0  \\ i & 0 & 0& 0 & 0& 0 \\ j & 0 & 0& 0 & 0& 0 \\
\end{array}\qquad
\begin{array}{c|ccccccccc} x & x^{*}   \\ \hline 0 & 0 \\ g & i\\ h & j \\i & g \\j & h \\  \end{array} \;\;\mathrm{and}\; 1=0.
\end{eqnarray*}
\end{example}

Let $\mathbf{Q}$ be a QB-algebra. We say that an element $x\in Q$ is $\emph{regular}$, if $x\vee x=x$ (or $x\wedge x=x$). And the set of all regular elements in $\mathbf{Q}$ is denoted by $\mathscr{R}(Q)$. Then $\mathscr{R}(\mathbf{Q})=\langle \mathscr{R}(Q); \vee ,\wedge ,^\ast, 0, 1 \rangle$ is a Boolean algebra where the operations $\vee,\wedge$, and $^\ast$ are those operations of $\mathbf{Q}$ restricted to $\mathscr{R}(Q)$. If $x\in Q\backslash\mathscr{R}(Q)$, then $\emph{x}$ is called $\emph{irregular}$ and the set of all irregular elements in $\mathbf{Q}$ is denoted by $\mathscr{IR}(Q)$. In addition, we define a binary relation $x\leq y$ iff $x \vee y=y\vee y$, or equivalently, $x \wedge y = x \wedge x$ for any $x,y\in Q$. Then the relation $\leq$ is quasi-ordering \cite{Engineering}. Moreover, for any $x\in Q$, we have $0\leq x\leq 1$ and $x\wedge x \leq x \leq x \vee x$. If $x\leq y$ and $y\leq x$, then $x\wedge y=x\wedge x$ and $y\wedge x=y\wedge y$, so $x\wedge x=y\wedge y$. Especially, if $x,y\in \mathscr{R}(Q)$ with $x\leq y$ and $y\leq x$, then $x=y$.

\begin{lemma}\label{Lemma 1}\emph{\cite{QB}}
Let $\mathbf{Q}$  be a QB-algebra and $x,y,u,v\in Q$. Then

$\mathrm{(1)}$ $\mathrm{if}$ $x\leq y$ $\mathrm{and}$ $u\leq v$, $\mathrm{then}$ $x\wedge u\leq y\wedge v$ $\mathrm{and}$ $\, x\vee u\leq y\vee v$;

$\mathrm{(2)}$ $1^{\ast}=0$ $\mathrm{and}$ $0^{\ast}=1$;

$\mathrm{(3)}$ $\mathrm{if}$ $x \leq y$, $\mathrm{then}$ $y^{\ast}\leq  x^{\ast}$.
\end{lemma}

Let $\mathbf{Q}$  be a QB-algebra. For any $x,y\in Q$, authors defined two relations in \cite{QB} as follows
\begin{align}
&\chi=\{\langle x,y\rangle \in Q^{2} \mid x\leq y\; \mathrm{and}\; y\leq x\};\cr
&\tau=\{\langle x,y\rangle \in Q^{2} \mid x=y \;\mathrm{or}\; x,y\in\mathscr{R}(Q) \}. \nonumber
\end{align}

\begin{lemma}\label{Lemma 2}\emph{\cite{QB}}
Let $\mathbf{Q}$ be a QB-algebra and $x,y\in Q$. Then

$\mathrm{(1)}$ $\chi$ and $\tau$ are congruences on $\mathbf{Q}$;

$\mathrm{(2)}$ $\langle x,y \rangle \in \chi \;\mathrm{iff}\; x\vee x=y\vee y$;

$\mathrm{(3)}$ the quotient algebra $\mathbf{Q}/ \chi= \langle Q/ \chi; \vee^{\mathbf{Q}/ \chi},\wedge^{\mathbf{Q}/ \chi}, ^{\ast{\mathbf{Q}/ \chi}}, 0/\chi, 1/\chi  \rangle $ is a Boolean algebra where $Q/ \chi=\{x/ \chi\,|\,x\in Q\}$;

$\mathrm{(4)}$ the quotient algebra $\mathbf{Q}/ \tau= \langle Q/ \tau; \vee^{\mathbf{Q}/ \tau},\wedge^{\mathbf{Q}/ \tau}, ^{\ast{\mathbf{Q}/ \tau}}, 0/\tau, 1/\tau  \rangle $ is a flat QB-algebra where $ Q/ \tau=\{x/ \tau\,|\,x\in Q\}$.

\end{lemma}
\begin{remark}
The set $x/\chi=\{y\in Q \,|\,\langle x,y\rangle\in \chi\}$ is called a $cloud$ in $\mathbf{Q}$.
\end{remark}

Let $\mathbf{Q}_{1}$ and $\mathbf{Q}_{2}$ be QB-algebras. A $\emph{QB-homomorphism}$ is a mapping $f: Q_{1}\rightarrow Q _{2}$ such that $f(x\vee y)=f(x) \vee f(y)$, $f(x\wedge y)=f(x) \wedge f(y)$, $f(x^{\ast})=(f(x))^{\ast}$, and $f(0)=0$ for any $x,y\in Q_{1}$. A $\emph{QB-embedding}$ is an injective QB-homomorphism, a $\emph{QB-epimorphism}$ is a surjective QB-homomorphism, and a $\emph{QB-isomorphism}$ is a bijective QB-homomorphism. If there exists a QB-isomorphism between QB-algebras $\mathbf{Q}_{1}$ and $\mathbf{Q}_{2}$, then we call that $\mathbf{Q}_{1}$ is \emph{isomorphic} to $\mathbf{Q}_{2}$.

\begin{remark}
Let $\mathbf{4}$ be the QB-algebra defined in Example $\ref{Example 4}$ and $\bar{\mathbf{4}}$ be the QB-algebra defined in Example $\ref{Example 4-}$. Then $\mathbf{4}$ is isomorphic to $\bar{\mathbf{4}}$.
\end{remark}

\begin{theorem}\label{Theorem 1}\emph{\cite{QB}}
Let $\mathbf{Q}$ be a QB-algebra. Then there exist a Boolean algebra $\mathbf{Q}/ \chi$ and a flat QB-algebra $\mathbf{Q}/ \tau$ such that $\mathbf{Q}$ can be embedded into the direct product $\mathbf{Q}/ \chi \times \mathbf{Q}/ \tau$.
\end{theorem}

\begin{example}\label{Example 3}
Let $\mathbf{4}$ be the QB-algebra defined in Example $\ref{Example 4}$. Then $\mathbf{4}/ \chi$ is the Boolean algebra of two elements, which is isomorphic to $\mathbf{2}$ and $\mathbf{4}/\tau$ is the flat QB-algebra, which is isomorphic to $\mathbf{F}_{3}$ defined in Example $\ref{Example F3}$. Hence $\mathbf{4}$ can be embedded into the direct product $\mathbf{2}\times \mathbf{F}_{3}$.
\end{example}

\section{Finite quasi-Boolean algebras}

In this section, we discuss the number of finite quasi-Boolean algebras. For the sake of clarity, we discuss two cases: flat QB-algebras and non-flat QB-algebras.

\begin{lemma}\label{Lemma Cloud}
Let $\mathbf{Q}$ be a QB-algebra. Then any cloud in $\mathbf{Q}$ contains one and only one regular element.
\end{lemma}
\begin{proof}
For any $x\in Q$, consider the cloud $x/\chi$, since $x\leq x\vee x$ and $x\vee x=x\wedge x\leq x$, we have $\langle x ,x \vee x\rangle \in \chi$ and then $x \vee x \in x/ \chi$. Hence for any cloud $x/ \chi$, it contains the regular element $x\vee x$. Now, suppose that $\emph{y}$ is a regular element in $x/ \chi$. Then we have $x\vee x=y\vee y$ by Lemma \ref{Lemma 2}(2). Since $\emph{y}$ is regular, we have $y=y\vee y=x\vee x$. Thus the cloud $x/ \chi$ contains one and only one regular element.
\end{proof}

\begin{remark}
Let $x/\chi$ be a cloud in $\mathbf{Q}$ and $a\in x/\chi$ be the regular element in $\mathbf{Q}$. Then $x/\chi=a/\chi$ and we denote $cl(a)=x/\chi$. In the following, the notation  $cl(x)$ means $x/\chi$ where $\emph{x}$ is a regular element in $\mathbf{Q}$.
\end{remark}

\begin{lemma}\label{Lemma number}
Let $\mathbf{Q}$ be a QB-algebra and $x,y\in Q$. Then

\emph{(1)} if there are $y,y^{\ast}\in Q$ such that $y,y^{\ast} \in cl(x)$, then $x=x^{\ast}$;

\emph{(2)} if $y\in cl(x)$, then $y^{\ast}\in cl(x^{\ast})$;

\emph{(3)} if $x=x^{\ast}$ and $y\in cl(x)$, then $y^{\ast}\in cl(x)$.
\end{lemma}

\begin{proof}
(1) If $y,y^{\ast} \in cl(x)$, then we have $x\vee x=y\vee y$ and $x\vee x=y^{\ast} \vee y^{\ast}$. Since $x=x\vee x$ and $y\vee y=y\wedge y$, we have $x=x\vee x=y^{\ast}\vee y^{\ast}=(y\wedge y)^{\ast}=(y\vee y)^{\ast}=(x\vee x)^{\ast}=x^{\ast}$ by (QB4).

(2) If $y\in cl(x)$, then we have $x\vee x=y\vee y$, it follows that $x^{\ast}\vee x^{\ast}=(x\wedge x)^{\ast}=(x\vee x)^{\ast}=(y\vee y)^{\ast}=y^{\ast} \wedge y^{\ast}=y^{\ast} \vee y^{\ast}$, so $y^{\ast}\in cl(x^{\ast})$.

(3) Since $x=x^{\ast}$, we have $y^{\ast}\in cl(x)$ by (2).
\end{proof}

Given a QB-algebra $\mathbf{Q}$ and a subset $\emph{X}$ of $\emph{Q}$, we denote $X^{\ast}=\{x^{\ast}\in Q\,|\,x\in X\}$. Especially, we have $(cl(x))^{\ast}=\{y^{\ast}\in Q\,|\,y\in cl(x)\}$. Below we show that the set $(cl(x))^{\ast}$ is the cloud $cl(x^{\ast})$.

\begin{lemma}\label{clx}
Let $\mathbf{Q}$ be a QB-algebra and $x\in \mathscr{R}(Q)$. Then $(cl(x))^{\ast}=cl(x^{\ast})$.
\end{lemma}

\begin{proof}
By Lemma \ref{Lemma number}(2), we have $(cl(x))^{\ast}\subseteq cl(x^{\ast})$. On the other hand, for any $y\in cl(x^{\ast})$, we have $y\vee y=x^{\ast} \vee x^{\ast}$, it turns out that $y^{\ast}\vee y^{\ast}=(y\wedge y)^{\ast}=(y\vee y)^{\ast}=(x^{\ast}\vee x^{\ast})^{\ast}=x\wedge x=x\vee x$, so $y^{\ast}\in cl(x)$. Since $y=y^{\ast\ast}\in (cl(x))^{\ast}$, we have  $cl(x^{\ast})\subseteq (cl(x))^{\ast}$. Thus $(cl(x))^{\ast}=cl(x^{\ast})$.
\end{proof}

\begin{lemma}\label{bijection}
Let $\mathbf{Q}$ be a QB-algebra. For any $x\in \mathscr{R}(Q)$, we define a mapping $f: cl(x) \rightarrow cl(x^{\ast})$ by $f(y)=y^{\ast}$ for any $y\in cl(x)$. Then f is a bijection.
\end{lemma}

\begin{proof}
For any $z \in cl(x^{\ast})$, since $cl(x^{\ast})=(cl(x))^{\ast}$ by Lemma \ref{clx}, we have $z=y^{\ast}$ for some $y\in cl(x)$, so $\emph{f}$ is surjective. Suppose that $f(y_{1})=f(y_{2})$ for any $y_{1},y_{2}\in cl(x)$. Then we have $y_{1}^{\ast}=y_{2}^{\ast}$, which implies that $y_{1}=y_{1}^{\ast\ast}=y_{2}^{\ast\ast}=y_{2}$ by (QB5), so $\emph{f}$ is injective. Hence $\emph{f}$ is a bijection.
\end{proof}

\begin{lemma}\label{neq}
Let $\mathbf{Q}$ be a non-flat QB-algebra. Then $x^{\ast}\neq x$ for any $x\in Q$.
\end{lemma}

\begin{proof}
Suppose that there exists an element $x\in Q$ such that $x^{\ast}=x$. Then $1=x\vee x^{\ast}=x\vee x=x\wedge x=x\wedge x^{\ast}=0$, this is a contradiction with the hypothesis. Thus $x^{\ast}\neq x$ for any $x\in Q$.
\end{proof}

\begin{corollary}\label{cor empty}
Let $\mathbf{Q}$ be a non-flat QB-algebra and $x\in \mathscr{R}(Q)$. Then $cl(x)\cap cl(x^{\ast})=\emptyset$.
\end{corollary}

\begin{proof}
Suppose that there exists an element $y\in cl(x)\cap cl(x^{\ast})$. Then $y\vee y=x\vee x=x$ and $y\vee y=x^{\ast}\vee x^{\ast}=x^{\ast}$, it turns out that $x=x^{\ast}$, this is a contradiction. Thus $cl(x)\cap cl(x^{\ast})=\emptyset$.
\end{proof}

\begin{remark}
Let $\mathbf{Q}$ be a non-flat QB-algebra. The result of Corollary \ref{cor empty} can be extended to broader contexts, i.e., $cl(x_{i})\cap cl(x_{j})=\emptyset$ for any $x_{i},x_{j}\in \mathscr{R}(Q)$ with $x_{i}\neq x_{j}$.
\end{remark}

\begin{proposition}\label{non flat number}
Let $\mathbf{Q}$ be a non-flat QB-algebra. Then

\emph{(1)} if $|\mathscr{R}(Q)|$ is finite, then $|\mathscr{R}(Q)|$ is even and $Q=cl(x_{1})\sqcup cl(x_{2})\sqcup \cdots \sqcup cl(x_{k})\sqcup cl(x_{1}^{\ast})\sqcup cl(x_{2}^{\ast})\sqcup \cdots \sqcup cl(x_{k}^{\ast})$ where $x_{i}\neq x_{j}$$(i\neq j)$, $x_{i}\neq x_{j}^{\ast}$$(i\neq j)$, and $\sqcup$ is the disjoint union;

\emph{(2)} if $|Q|$ is finite, then $|Q|$ is even.
\end{proposition}

\begin{proof}
(1) Since $|\mathscr{R}(Q)|$ is finite, we can suppose that $\mathscr{R}(Q)=\{x_{1},x_{2},\cdots, x_{n}\}$ where the elements are all distinct. For any $x_{i}\in \mathscr{R}(Q)$, we have $x_{i}^{\ast}\in \mathscr{R}(Q)$ and $x_{i}^{\ast}\neq x_{i}$ by Lemma \ref{neq}. Then there exists $x_{j}\in \mathscr{R}(Q)$ such that $x_{j}=x_{i}^{\ast}$$(j\neq i)$, so $\mathscr{R}(Q)$ can be expressed as $\{x_{1}, x_{2}, \cdots, x_{k}, x_{1}^{\ast}, x_{2}^{\ast}, \cdots, x_{k}^{\ast}\}$ where $x_{j}^{\ast}\neq x_{i}$ $(i,j=1,2,\cdots,k)$ and $n=2k$. Hence $|\mathscr{R}(Q)|=2k$ is even. Moreover, on the one hand, $cl(x_{1})\sqcup \cdots \sqcup cl(x_{k}^{\ast})\subseteq Q$. On the other hand, for any $x\in Q$, since $x\vee x\in \mathscr{R}(Q)$, we have $x\in cl(x\vee x)\subseteq cl(x_{1})\sqcup \cdots \sqcup cl(x_{k}^{\ast})$, so $Q\subseteq cl(x_{1})\sqcup \cdots \sqcup cl(x_{k}^{\ast})$. Hence $cl(x_{1})\sqcup \cdots \sqcup cl(x_{k}^{\ast})=Q$.

(2) For any $x_{i}\in \mathscr{R}(Q)$, there exists $x_{i}^{\ast}\in \mathscr{R}(Q)$ such that $|cl(x_{i})|=|cl(x_{i}^{\ast})|$ by Lemma \ref{bijection}. Since $|Q|$ is finite, we have that $|\mathscr{R}(Q)|$ is finite, it turns out that $|Q|=2\sum_{i=1}^k|cl(x_{i})|$, so $|Q|$ is even.
\end{proof}

\begin{definition}\label{Definition 4}
Let $\mathbf{Q}$ be a non-flat QB-algebra. Then $\mathbf{Q}$ is called \emph{irreducible}, if $\mathscr{R}(Q)=\{0,1\}$.
\end{definition}

\begin{example}
Let $\mathbf{4}$ be the QB-algebra defined in Example $\ref{Example 4}$. Then $\mathscr{R}(4)=\{0,1\}$ and then $\mathbf{4}$ is irreducible. However, the QB-algebra $\mathbf{A}$ defined in Example $\ref{Example A}$, is not irreducible, since $\mathscr{R}(A)=\{0,a, f,1\}$.
\end{example}

\begin{corollary} Let $\mathbf{Q}$ be an irreducible QB-algebra. Then $Q=cl(0)\sqcup cl(1)$. Moreover, if $|Q|$ is finite, then $|Q|$ is even.
\end{corollary}

\begin{corollary}\label{c3.3}
Let $\mathbf{Q}$ be an irreducible QB-algebra. Then $\mathbf{Q}/\chi$ is isomorphic to $\mathbf{2}$.
\end{corollary}

\begin{proposition}
Let $\mathbf{Q}$ be a non-flat QB-algebra. Then $\mathbf{Q}$ is irreducible if and only if $\mathbf{Q}$ is embedded into the direct product of $\mathbf{2}$ and $\mathbf{F}_{2k+1}$
where $\mathbf{F}_{2k+1}$ is some flat QB-algebra and $k\in \mathbb{N}$.
\end{proposition}

\begin{proof}
Let $\mathbf{Q}$ be irreducible. Then by Corollary \ref{c3.3} and Theorem \ref{Theorem 1}, we have that $\mathbf{Q}$ is embedded into the direct product of $\mathbf{2}$ and $\mathbf{Q}/\tau$.
Note that $x/\tau=\{x\}$ for $x\in \mathscr{IR}(Q)$ and $0/\tau=1/\tau=\{0,1\}$, we get that $\mathbf{Q}/\tau=\{0/\tau, x_{1}/\tau, x^{\ast}_{1}/\tau,\cdots,x_{i}/\tau, x^{\ast}_{i}/\tau,\cdots\}$ for some $i\in \mathbb{N}$. Denote $\mathbf{F}_{2k+1}=\mathbf{Q}/\tau$. Then $\mathbf{Q}$ is embedded into the direct product of $\mathbf{2}$ and $\mathbf{F}_{2k+1}$.
Conversely, since the direct product of $\mathbf{2}$ and $\mathbf{F}_{2k+1}$ is a QB-algebra, we have $\mathscr{R}(\mathbf{2}\times \mathbf{F}_{2k+1})=\{\langle 0,0\rangle,\langle 1,0\rangle\}$. Because $\mathbf{Q}$ is embedded into the QB-algebra $\mathbf{2}\times \mathbf{F}_{2k+1}$, we have $|\mathscr{R}(Q)|\leq 2$. Moreover, $0, 1\in \mathscr{R}(Q)$, it turns out that $\mathscr{R}(Q)=\{0,1\}$, so $\mathbf{Q}$ is irreducible.
\end{proof}

\begin{corollary}
Let $\mathbf{Q}$ be a finite non-flat QB-algebra. Then $\mathbf{Q}$ is irreducible if and only if $\mathbf{Q}$ is the direct product of $\mathbf{2}$ and $\mathbf{F}_{2k+1}$
where $\mathbf{F}_{2k+1}$ is a finite flat QB-algebra defined as Example $\ref{Example F3}$ and $k\in \mathbb{N}$. Moreover, $|Q|=4k+2$.
\end{corollary}

Below we see the case of flat QB-algebras.

\begin{lemma}\label{Lemma 3}\label{vee 0}
Let $\mathbf{Q}$ be a flat QB-algebra and $x,y\in Q$. Then $x\wedge y=x\vee y=0$.
\end{lemma}

\begin{corollary}\label{Cor RQ}
Let $\mathbf{Q}$ be a flat QB-algebra. Then $\mathscr{R}(Q)=\{0\}$.
\end{corollary}

\begin{proof}
Obviously, $0\in \mathscr{R}(Q)$. For any $x\in \mathscr{R}(Q)$, we have $x\vee x=x$. Since $x=x\vee x=0$ by Lemma \ref{Lemma 3}, we have $\mathscr{R}(Q)\subseteq\{0\}$. Hence $\mathscr{R}(Q)=\{0\}$.
\end{proof}

\begin{corollary}\label{cl0Q}
Let $\mathbf{Q}$ be a flat QB-algebra. Then $cl(0)=Q$.
\end{corollary}
\begin{proof}
Obviously, $cl(0)\subseteq Q$. On the other hand, for any $x\in Q$, since $x\vee x=0=0\vee 0$ by Lemma \ref{vee 0}, we have $x\in cl(0)$, so $Q\subseteq cl(0)$. Hence $Q=cl(0)$.
\end{proof}

\begin{proposition}
Let $\mathbf{Q}$ be a finite flat QB-algebra. If $\mathbf{Q}$ contains $k$ elements $x_{1}, \cdots, x_{k}$ with $x_{i}=x_{i}^{\ast} (i=1,\cdots, k)$, then $|Q|=m+k$ where $m$ is even. Especially, if $k=1$, then $|Q|$ is odd.
\end{proposition}

\begin{proof}
Denote $S=\{x\in Q\,|\,x^{\ast}=x\}$. Then $|S|=k$. Since $\mathbf{Q}$ is a finite flat QB-algebra, we have $0^{\ast}=1=0$, it turns out that $0\in S$, so $k\geq 1$. Denote $Q\backslash S=\{x\in Q|x^{\ast}\neq x\}$. For any $x\in Q\backslash S$, we have $x^{\ast}\in Q\backslash S$, so $|Q\backslash S|$ is even. Hence $|Q|=|Q\backslash S|+|S|=m+k$ where $\emph{m}$ is even. Especially, if $k=1$, i.e., $S=\{0\}$, then $|Q|$ is odd.
\end{proof}

\section{The standard completeness of QB-algebras}

In this section, we show the standard completeness of QB-algebras. At first, we prove the standard completeness theorem for $\mathbb{FQB}$ with respect to the flat QB-algebra $ \mathbf{F}_{3}$.

\begin{theorem}\label{Proposition 4.1}
Let $p$ and $q$ be terms in the language of QB-algebras. Then
$$
\mathbb{FQB}\models p\approx q  \; \mathrm{iff}  \; \mathrm{\mathbf{F}_{3}}\models p\approx q.
$$
\end{theorem}

\begin{proof}
Suppose that $\mathbb{FQB}\models p\approx q$. Then we have $\mathrm{\mathbf{F}_{3}}\models p\approx q$.

Conversely, suppose that $\mathrm{\mathbf{F}_{3}}\models p\approx q$ and if $\mathbb{FQB}\nvDash p\approx q$, then there exists a flat QB-algebra $\mathbf{F}$ such that $\mathrm{\mathbf{F}}\nvDash p\approx q$. We distinguish several cases.

(1) Both $p(x_{1},x_{2},\cdots,x_{n})$ and $q(x_{1},x_{2},\cdots,x_{m})$ contain at least an occurrence of a binary operation. Then for any $a_{1},a_{2},\cdots,a_{m},\cdots,a_{n}\in F$, we have that
$$
p^\mathbf{F}(a_{1},a_{2},\cdots,a_{n})=0^\mathbf{F}=q^\mathbf{F}(a_{1},a_{2},\cdots,a_{m})
$$
from Lemma \ref{Lemma 3}, against the hypothesis.

(2) Either $p$ or $q$ contains at least an occurrence of a binary operation. Without loss of generality, we assume that $p$ contains at least an occurrence of a binary operation and $q$ does not contain any occurrence of a binary operation. Then $q$ is one of the following forms.

(2.1) If $q$ is a variable $\emph{x}$ and contains $k\,(0\leq k)$ occurrences of the operation $^{*}$, then we can assign $x$ the value $c\in \mathbf{F}_{3}$ and the variables $x_{1},\cdots ,x_{n}$ any values $0,c,d\in \mathbf{F}_{3}$, then $p^\mathbf{F}_{3}(0,\cdots,c,\cdots,d)=0$ $\neq $ $q^\mathbf{F}_{3}(c)$, this is a contradiction with $\mathrm{\mathbf{F}_{3}}\models p\approx q$.

(2.2) If $q$ is a constant and contains $k\,(0\leq k)$ occurrences of the operation $^{*}$, then $q^\mathbf{F}=0$, it follows that $q^\mathbf{F}=0=p^\mathbf{F}$, against the hypothesis.

(3) Neither $p$ nor $q$ contains any occurrence of a binary operation, then $p$ and $q$ contain at most one variable, it turns out that they have one of the following forms.

(3.1) $p$ is a constant and contains $k \,(0\leq k)$ occurrences of the operation $^{*}$, $q$ is a constant and contains $l \, (0\leq l)$ occurrences of the operation $^{*}$. Then we have $p^\mathbf{F}=0=q^\mathbf{F}$, against the hypothesis.

(3.2) $p$ is a variable $\emph{x}$ and contains $k\,(0\leq k)$ occurrences of the operation $^{*}$, $q$ is a constant and contains $l\, (0\leq l)$ occurrences of the operation $^{*}$. Assign $\emph{x}$ the value $c\in \mathbf{F}_{3}$, we can get that $p^\mathbf{F}_{3}(c)\neq 0=q^\mathbf{F}_{3}$, this is a contradiction with
$\mathrm{\mathbf{F}_{3}}\models p\approx q$.

(3.3) $p$ is a variable $x$ and contains $k\,(0\leq k)$ occurrences of the operation $^{*}$, $q$ is a variable $y$ and contains $l\,(0\leq l)$ occurrences of the operation $^{*}$. Assign $x$ the value $c\in \mathbf{F}_{3}$ and $y$ the value $0\in \mathbf{F}_{3}$, it follows that $p^\mathbf{F}_{3}(c)\neq q\mathbf{F}_{3}(0)$, this is a contradiction with $\mathrm{\mathbf{F}_{3}}\models p\approx q$.

(3.4) $\emph{p}$ is a variable $\emph{x}$ and contains $k\,(0\leq k)$ occurrences of the operation $^{*}$, $\emph{q}$ is a variable $\emph{x}$ and contains $l\,(0\leq l)$ occurrences of the operation $^{*}$. If $\emph{k}$ and $\emph{l}$ have same parity, then $p^\mathbf{F}$=$q^\mathbf{F}$, against the hypothesis. If $\emph{k}$ and $\emph{l}$ have different parities, then we can assign $\emph{x}$ the value $c\in \mathbf{F}_{3}$, it follows that $p^\mathbf{F}_{3}(c)\neq q^\mathbf{F}_{3}(c)$, this is a contradiction with $\mathrm{\mathbf{F}_{3}}\models p\approx q$.

Thus if $\mathrm{\mathbf{F}_{3}}\models p\approx q$, then we have $\mathbb{FQB}\models p\approx q$.
\end{proof}

\begin{remark}\label{FQB variety}
By Theorem \ref{Proposition 4.1}, we have that the flat QB-algebra $\mathrm{\mathbf{F}_{3}}$ generates $\mathbb{FQB}$.
\end{remark}

\begin{lemma}\label{Lemma 4.2}\emph{\cite{course}}
Let $p$ and $q$ be terms in the language of QB-algebras. Then
$$
\mathbb{B}\models p\approx q  \; \mathrm{iff}  \;  \mathbf{2}\models p \approx q.
$$
\end{lemma}

Next we prove the standard completeness theorem for $\mathbb{QB}$ with respect to the QB-algebra $\mathbf{4}$.

\begin{theorem}\label{Theorem 2}
Let $p$ and $q$ be terms in the language of QB-algebras. Then
$$
\mathbb{QB}\models p\approx q  \; \mathrm{iff} \; \mathbf{4}\models p \approx q.
$$
\end{theorem}

\begin{proof}
If $\mathbb{QB}\models p\approx q$, then we have $\mathrm{\mathbf{4}}\models p\approx q$.
Conversely, if $\mathrm{\mathbf{4}}\models p\approx q$ and we suppose that $\mathbb{QB}\nvDash p\approx q$, then there exists a QB-algebra $\mathrm{\mathbf{Q}}$ such that $\mathrm{\mathbf{Q}}\nvDash p\approx q$. By Theorem $\ref{Theorem 1}$, we have that $\mathrm{\mathbf{Q}}$ is embedded into the direct product of the Boolean algebra $\mathbf{Q}/ \chi$ and the flat QB-algebra $\mathbf{Q}/ \tau$. If $\mathrm{\mathbf{Q}}\nvDash p\approx q$, then either $\mathbf{Q}/ \chi\nvDash p\approx q$ or $\mathbf{Q}/ \tau\nvDash p\approx q$. If the former holds, then we have $\mathrm{\mathbf{2}}\nvDash p\approx q$ by Lemma \ref{Lemma 4.2}, so $\mathrm{\mathbf{4}}\nvDash p\approx q$. If the latter holds, then we have $\mathrm{\mathbf{F}_{3}}\nvDash p\approx q$ by Theorem \ref{Proposition 4.1}, so $\mathrm{\mathbf{4}}\nvDash p\approx q$. Thus if $\mathrm{\mathbf{4}}\models p\approx q$, then we have $\mathbb{QB}\models p\approx q$.
\end{proof}

\begin{remark}\label{FQB variety}
Similarly to Theorem \ref{Theorem 2}, we have that
$\mathbb{QB}\models p\approx q\; \mathrm{iff} \; \bar{\mathbf{4}}\models p \approx q$.
Moreover, we can see that the QB-algebras $\mathbf{4}$ and $\bar{\mathbf{4}}$ can generate the variety $\mathbb{QB}$, respectively.
\end{remark}

\section{Congruences on QB-algebras}

In this section, we show the congruence extension of the variety $\mathbb{QB}$. A variety $\mathbb{D}$ has the \emph{congruence extension property} (CEP, for short) if
for any $\mathbf{D}\in \mathbb{D}$, any subalgebra $\mathbf{D_{0}}$ of $\mathbf{D}$, and any congruence $\theta^{\mathbf{D_{0}}}$ on $\mathbf{D_{0}}$, there exists a congruence $\theta^{\mathbf{D}}$ on $\mathbf{D}$ which extends $\theta^{\mathbf{D_{0}}}$, i.e., $\theta^{\mathbf{D_{0}}} = \theta^{\mathbf{D}} \cap D_{0}^{2}$.

Based on the subdirect product decomposition of a QB-algebra, we can transform the study of the CEP in the variety $\mathbb{QB}$ into the study of the CEPs in the varieties $\mathbb{B}$ and $\mathbb{FQB}$, respectively.

\begin{lemma}\label{B CEP}\emph{\cite{Boolean CEP}}
The variety $\mathbb{B}$ has the CEP.
\end{lemma}

Since the variety $\mathbb{B}$ has the CEP, we next discuss the variety $\mathbb{FQB}$.

\begin{lemma}\label{Lemma 5}
The variety $\mathbb{FQB}$ has the CEP.
\end{lemma}

\begin{proof}
For any $\mathbf{F}\in \mathbb{FQB}$, $\mathbf{F_{0}}$ is any subalgebra of $\mathbf{F}$, and $\theta$ is any congruence on $\mathbf{F_{0}}$, we define a binary relation $\theta' =\{\langle x,y \rangle \in F^{2} \,|\, \langle x,y \rangle \in \theta \,\mathrm{or}\,x=y  \}$. Then $\theta'$ is a congruence on $\mathbf{F}$ such that $\theta =\theta' \cap F_{0} ^{2}$. Obviously, $\theta'$ is an equivalence relation on $\mathbf{F}$. For any $\langle x,y \rangle, \langle u,v \rangle \in \theta '$, since $x\vee u=0=y\vee v$ and $x\wedge u=0=y\wedge v$ by Lemma \ref{Lemma 3}, we have $\langle x\vee u, y\vee v\rangle \in \theta'$ and $\langle x\wedge u, y\wedge v\rangle \in \theta'$. For any $\langle x,y\rangle \in \theta'$, we have $\langle x,y\rangle \in \theta$ or $x=y$. If $\langle x, y\rangle \in \theta$, then we have $\langle x^{\ast}, y^{\ast}\rangle \in \theta$, so $\langle x^{\ast}, y^{\ast}\rangle \in \theta'$. If $x=y$, then we have $x^{\ast}=y^{\ast}$, so $\langle x^{\ast}, y^{\ast}\rangle \in \theta'$. Hence $\theta'$ is a congruence on $\mathbf{F}$. Moreover, for any $\langle x,y \rangle \in \theta$, we have $\langle x,y \rangle \in F_{0}^{2}$ and $\langle x,y \rangle \in \theta'$, which imply that $\langle x,y \rangle \in \theta' \cap F_{0}^{2}$, so $\theta \subseteq \theta' \cap F_{0} ^{2}$. Conversely, for any $\langle x,y \rangle \in \theta' \cap F_{0}^{2}$, then we have $\langle x,y \rangle \in \theta'$ and $\langle x,y \rangle \in F_{0}^{2}$, it turns out that $\langle x,y \rangle \in \theta$ or $x=y$. If $\langle x,y \rangle \in \theta$, then $\theta' \cap F_{0}^{2}\subseteq \theta$. If $x=y$, since $\theta$ is a congruence on $\mathbf{F_{0}}$, we have $\langle x,y \rangle \in \theta$ and then $\theta' \cap F_{0}^{2}\subseteq \theta$. So $\theta =\theta' \cap F_{0}^{2}$. Thus the variety $\mathbb{FQB}$ has the CEP.
\end{proof}

\begin{lemma}\label{Lemma 6}
Let $\mathbf{Q}$ be a QB-algebra and $\theta$ be a congruence on $\mathbf{Q}$. Then there exist a congruence $\theta_{1}$ on $\mathbf{Q}/\chi$ and a congruence $\theta_{2}$ on $\mathbf{Q}/\tau$ such that $\langle x,y \rangle \in \theta$ iff $\langle \langle x/\chi ,x/\tau \rangle, \langle y/\chi ,y/\tau \rangle \rangle \in \theta_{1} \times \theta_{2}$ for any $x,y\in Q$.
\end{lemma}

\begin{proof}
Let $\mathbf{Q}$ be a QB-algebra and $\theta$ be a congruence on $\mathbf{Q}$. Define a binary relation $\theta'=$$\{\langle \langle x/\chi , x/\tau\rangle, \\ \langle y/\chi , y/\tau\rangle \rangle \,|\, \langle x,y \rangle \in \theta \}$. Then $\theta'$ is a congruence on $\mathbf{Q}/\chi \times \mathbf{Q}/\tau$ such that $\langle \langle x/\chi , x/\tau\rangle, \langle y/\chi , y/\tau\rangle \rangle \in \theta'$ iff $\langle x,y\rangle \in \theta$. Moreover, we define two binary relations $\theta_{1}$ on $\mathbf{Q}/\chi$ and $\theta_{2}$ on $\mathbf{Q}/\tau$ by $\langle x/ \chi , y/ \chi \rangle \in \theta_{1}\;\mathrm{iff}\; \langle \langle x/ \chi,x/\tau \rangle , \langle y/ \chi,y/\tau \rangle \rangle \in \theta'$ and $\langle x/ \tau , y/ \tau \rangle \in \theta_{2} \;\mathrm{iff}\; \langle \langle x/ \chi,x/\tau \rangle , \langle y/ \chi,y/\tau \rangle \rangle\in \theta'$, respectively. Then $\theta_{1}$ is the congruences on $\mathbf{Q}/\chi$ and  $\theta_{2}$ is the congruence on $\mathbf{Q}/\tau$. In addition, we have that $\langle x,y \rangle \in \theta$ iff $\langle \langle x/ \chi,x/\tau \rangle , \langle y/ \chi,y/\tau \rangle \rangle \in \theta'$ iff $\langle \langle x/\chi ,x/\tau \rangle,  \langle y/\chi ,y/\tau \rangle \rangle \in \theta_{1} \times \theta_{2}$.
\end{proof}

\begin{lemma}\label{Lemma 7}
Let $\mathbf{Q}$ be a QB-algebra and $\mathbf{Q_{0}}$ be a subalgebra of $\mathbf{Q}$. Then we have $\chi ^{\mathbf{Q_{0}}}=\chi ^{\mathbf{Q}} \cap Q_{0}^{2}$ and $\tau ^{\mathbf{Q_{0}}}=\tau ^{\mathbf{Q}} \cap Q_{0}^{2}$.
\end{lemma}

\begin{proof}
Let $\mathbf{Q}$ be a QB-algebra and $\mathbf{Q_{0}}$ be a subalgebra of $\mathbf{Q}$. For any $\langle x,y\rangle \in \chi^{\mathbf{Q_{0}}}$, we have $x\vee x=y\vee y$ and $x,y \in Q_{0}$. Note that $\mathbf{Q_{0}}$ is a subalgebra of $\mathbf{Q}$, then $x,y \in Q$ and $\langle x,y\rangle \in \chi^{\mathbf{Q}}$, it turns out that $\langle x,y\rangle \in \chi^{\mathbf{Q}} \cap Q_{0}^{2}$, so $\chi^{\mathbf{Q_{0}}}\subseteq \chi^{\mathbf{Q}}\cap Q_{0}^{2}$. Conversely, for any $\langle x,y\rangle \in \chi^{\mathbf{Q}} \cap Q_{0}^{2}$, we have $\langle x,y \rangle \in Q_{0}^{2}$ and $x\vee x=y\vee y$, it turns out that $\langle x,y\rangle \in \chi^{\mathbf{Q_{0}}}$, so $\chi^{\mathbf{Q}}\cap Q_{0}^{2}\subseteq \chi^{\mathbf{Q_{0}}}$. Thus $\chi^{\mathbf{Q_{0}}}=\chi^{\mathbf{Q}}\cap Q_{0}^{2}$. Similarly, we can show that $\tau^{\mathbf{Q_{0}}}=\tau^{\mathbf{Q}}\cap Q_{0}^{2}$.
\end{proof}

\begin{lemma}\label{Lem subalgebra}
Let $\mathbf{Q}$ be a QB-algebra and $\mathbf{Q_0}$ be a subalgebra of $\mathbf{Q}$. Then $\mathbf{Q_0}/\chi^{\mathbf{Q_0}}$ is a subalgebra of $\mathbf{Q}/\chi^{\mathbf{Q}}$ and $\mathbf{Q_0}/\tau^{\mathbf{Q_0}}$ is a subalgebra of $\mathbf{Q}/\tau^{\mathbf{Q}}$.
\end{lemma}

\begin{proof}
We need to show that $Q_0/\chi^{\mathbf{Q_0}} \subseteq Q/\chi^{\mathbf{Q}}$ and the operations of $\mathbf{Q_0}/\chi^{\mathbf{Q_0}}$ are operations of $\mathbf{Q}/\chi^{\mathbf{Q}}$ restricted to $Q_0/\chi^{\mathbf{Q_0}}$. For any $x \in Q_0$ and $y \in x/\chi^{\mathbf{Q_0}}$, since $\chi^{\mathbf{Q_0}} = \chi^{\mathbf{Q}} \cap Q_0^{2}$ by Lemma \ref{Lemma 7}, we have $\langle x, y \rangle \in \chi^{\mathbf{Q}}$ and $y \in x/\chi^{\mathbf{Q}}$, it follows that $x/\chi^{\mathbf{Q_0}} \subseteq x/\chi^{\mathbf{Q}}$. Since $\mathbf{Q_0}$ is a subalgebra of $\mathbf{Q}$, we have $Q_0/\chi^{\mathbf{Q_0}} \subseteq Q/\chi^{\mathbf{Q}}$. For any $x/\chi^{\mathbf{Q_0}}, y/\chi^{\mathbf{Q_0}} \in Q_0/\chi^{\mathbf{Q_0}}$, we have $x/\chi^{\mathbf{Q_0}} \vee^{\mathbf{Q_0}/\chi} y/\chi^{\mathbf{Q_0}} = (x \vee y)/\chi^{\mathbf{Q_0}}$. Since $\mathbf{Q_0}$ is a subalgebra of $\mathbf{Q}$, we have $x/\chi^{\mathbf{Q_0}} \vee^{\mathbf{Q_0}/\chi} y/\chi^{\mathbf{Q_0}} = (x \vee y)/\chi^{\mathbf{Q_0}} = (x \vee y)/\chi^{\mathbf{Q}} = x/\chi^{\mathbf{Q}} \vee^{\mathbf{Q}/\chi} y/\chi^{\mathbf{Q}}$. Similarly, we can show that $x/\chi^{\mathbf{Q_0}} \wedge^{\mathbf{Q_0}/\chi} y/\chi^{\mathbf{Q_0}} = (x \wedge y)/\chi^{\mathbf{Q_0}} = (x \wedge y)/\chi^{\mathbf{Q}} = x/\chi^{\mathbf{Q}} \wedge^{\mathbf{Q}/\chi} y/\chi^{\mathbf{Q}}$. Meanwhile, $(x/\chi^{\mathbf{Q_0}})^{*{\mathbf{Q_0}/\chi}} = x^{*}/\chi^{\mathbf{Q_0}} = x^{*}/\chi^{\mathbf{Q}}$, $0/\chi^{\mathbf{Q_0}} = 0/\chi^{\mathbf{Q}}$, and $1/\chi^{\mathbf{Q_0}} = 1/\chi^{\mathbf{Q}}$. Thus $\mathbf{Q_0}/\chi^{\mathbf{Q_0}}$ is a subalgebra of $\mathbf{Q}/\chi^{\mathbf{Q}}$. Similarly, we can show that $\mathbf{Q_0}/\tau^{\mathbf{Q_0}}$ is a subalgebra of $\mathbf{Q}/\tau^{\mathbf{Q}}$.
\end{proof}

\begin{theorem}\label{QB CEP}
The variety $\mathbb{QB}$ has the CEP.
\end{theorem}

\begin{proof}
For any $\mathbf{Q}\in \mathbb{QB}$, $\mathbf{Q_{0}}$ is a subalgebra of $\mathbf{Q}$, and $\theta$ is a congruence on $\mathbf{Q_{0}}$. There exist a congruence $\theta_{1}$ on $\mathbf{Q_{0}}/\chi^{\mathbf{Q_{0}}}$ and a congruence $\theta_{2}$ on $\mathbf{Q_{0}}/\tau^{\mathbf{Q_{0}}}$ such that for any $x,y\in Q_{0}$, $\langle \langle x/\chi^{\mathbf{Q_{0}}} ,x/\tau^{\mathbf{Q_{0}}} \rangle, \langle y/\chi^{\mathbf{Q_{0}}} ,y/\tau^{\mathbf{Q_{0}}} \rangle \rangle \in \theta_{1} \times \theta_{2}$ iff  $\langle x,y \rangle \in \theta$ by Lemma \ref{Lemma 6}. Since the variety $\mathbb{B}$ has the CEP by Lemma \ref{B CEP} and the variety $\mathbb{FQB}$ also has the CEP by Lemma \ref{Lemma 5}, we have a congruence $\theta_{1}'$ on $\mathbf{Q}/ \chi ^{\mathbf{Q}}$ such that $\theta_{1}=\theta_{1}' \cap (Q_{0}/\chi ^{\mathbf{Q_{0}}})$ and a congruence $\theta_{2}'$ on $\mathbf{Q}/ \tau ^{\mathbf{Q}}$ such that $\theta_{2}=\theta_{2}' \cap (Q_{0}/\tau ^{\mathbf{Q_{0}}})$. Now, we only show that $\theta_{1} \times \theta_{2}=(\theta_{1}'\times \theta_{2}')\cap ((Q_{0}/ \chi ^{\mathbf{Q_{0}}}) \times (Q_{0}/ \tau ^{\mathbf{Q_{0}}}))$. For any $\langle \langle x/ \chi ^{\mathbf{Q_{0}}},x/ \tau ^{\mathbf{Q_{0}}} \rangle , \langle y/ \chi ^{\mathbf{Q_{0}}},y/ \tau ^{\mathbf{Q_{0}}} \rangle \rangle \in \theta_{1} \times \theta_{2} $, we have $\langle x/ \chi ^{\mathbf{Q_{0}}} , y/ \chi ^{\mathbf{Q_{0}}} \rangle \in \theta_{1}$ and $\langle x/ \tau ^{\mathbf{Q_{0}}} , y/ \tau ^{\mathbf{Q_{0}}} \rangle \in \theta_{2}$. Since $\theta_{1}=\theta_{1}' \cap (Q_{0}/\chi ^{\mathbf{Q_{0}}})$ and $\theta_{2}=\theta_{2}' \cap (Q_{0}/\tau ^{\mathbf{Q_{0}}})$, we have $\langle x/ \chi ^{\mathbf{Q_{0}}} , y/ \chi ^{\mathbf{Q_{0}}} \rangle \in \theta_{1}' \cap (Q_{0}/\chi ^{\mathbf{Q_{0}}})$ and $\langle x/ \tau ^{\mathbf{Q_{0}}} , y/ \tau ^{\mathbf{Q_{0}}} \rangle \in \theta_{2}' \cap (Q_{0}/\tau ^{\mathbf{Q_{0}}})$, it implies that $\langle x/ \chi ^{\mathbf{Q_{0}}} , y/ \chi ^{\mathbf{Q_{0}}} \rangle \in \theta_{1}'$, $\langle x/ \chi ^{\mathbf{Q_{0}}}, y/ \chi ^{\mathbf{Q_{0}}} \rangle \in Q_{0}/\chi ^{\mathbf{Q_{0}}}$, $\langle x/ \tau ^{\mathbf{Q_{0}}}, y/  \tau ^{\mathbf{Q_{0}}} \rangle \in \theta_{2}'$, and $\langle x/ \tau ^{\mathbf{Q_{0}}} , y/ \tau ^{\mathbf{Q_{0}}} \rangle \in Q_{0}/\tau ^{\mathbf{Q_{0}}}$, which means that $\langle \langle x/ \chi ^{\mathbf{Q_{0}}},x/ \tau ^{\mathbf{Q_{0}}} \rangle , \langle y/ \chi ^{\mathbf{Q_{0}}}, \\ y/ \tau ^{\mathbf{Q_{0}}} \rangle \rangle \in \theta_{1}' \times \theta_{2}' $ and $\langle \langle x/ \chi ^{\mathbf{Q_{0}}},x/ \tau ^{\mathbf{Q_{0}}} \rangle , \langle y/ \chi ^{\mathbf{Q_{0}}},y/ \tau ^{\mathbf{Q_{0}}} \rangle \rangle \in ((Q_{0}/\chi ^{\mathbf{Q_{0}}}) \times (Q_{0}/\tau ^{\mathbf{Q_{0}}})) $, so $\langle \langle x/ \chi ^{\mathbf{Q_{0}}},\\ x/ \tau ^{\mathbf{Q_{0}}} \rangle , \langle y/ \chi ^{\mathbf{Q_{0}}},y/ \tau ^{\mathbf{Q_{0}}} \rangle \rangle \in (\theta_{1}'\times \theta_{2}')\cap ((Q_{0}/ \chi ^{\mathbf{Q_{0}}}) \times (Q_{0}/ \tau ^{\mathbf{Q_{0}}})) $. Hence $\theta_{1} \times \theta_{2} \subseteq (\theta_{1}'\times \theta_{2}')\cap ((Q_{0}/ \chi ^{\mathbf{Q_{0}}}) \times (Q_{0}/ \tau ^{\mathbf{Q_{0}}}))$. Conversely, for any $\langle \langle x/ \chi ^{\mathbf{Q_{0}}},x/ \tau ^{\mathbf{Q_{0}}} \rangle, \langle y/ \chi ^{\mathbf{Q_{0}}},y/ \tau ^{\mathbf{Q_{0}}} \rangle \rangle  \in (\theta_{1}'\times \theta_{2}')\cap ((Q_{0}/ \chi ^{\mathbf{Q_{0}}}) \times (Q_{0}/ \tau ^{\mathbf{Q_{0}}}))$, we have $\langle \langle x/ \chi ^{\mathbf{Q_{0}}},x/ \tau ^{\mathbf{Q_{0}}} \rangle ,  \langle y/ \chi ^{\mathbf{Q_{0}}},y/ \tau ^{\mathbf{Q_{0}}} \rangle \rangle \in \theta_{1}'\times \theta_{2}'$ and $\langle \langle x/ \chi ^{\mathbf{Q_{0}}},x/ \tau ^{\mathbf{Q_{0}}} \rangle , \\ \langle y/ \chi ^{\mathbf{Q_{0}}},y/ \tau ^{\mathbf{Q_{0}}} \rangle \rangle \in ((Q_{0}/ \chi ^{\mathbf{Q_{0}}}) \times (Q_{0}/ \tau ^{\mathbf{Q_{0}}}))$, it turns out that $\langle  x/ \chi ^{\mathbf{Q_{0}}}, y/ \chi ^{\mathbf{Q_{0}}} \rangle \in \theta_{1}'$, $\langle  x/ \chi ^{\mathbf{Q_{0}}}, y/ \chi ^{\mathbf{Q_{0}}} \rangle \in Q_{0}/ \chi ^{\mathbf{Q_{0}}}$, $\langle  x/ \tau ^{\mathbf{Q_{0}}}, y/ \tau ^{\mathbf{Q_{0}}} \rangle$ $\in \theta_{2}'$, and $\langle  x/ \tau ^{\mathbf{Q_{0}}}, y/ \tau ^{\mathbf{Q_{0}}} \rangle \in Q_{0}/ \tau ^{\mathbf{Q_{0}}}$, which means that $\langle  x/ \chi ^{\mathbf{Q_{0}}},y/ \chi ^{\mathbf{Q_{0}}} \rangle \in \theta_{1}'\cap (Q_{0}/ \chi ^{\mathbf{Q_{0}}})$ and $\langle  x/ \tau ^{\mathbf{Q_{0}}}, y/ \tau ^{\mathbf{Q_{0}}} \rangle \in \theta_{2}'\cap (Q_{0}/ \tau ^{\mathbf{Q_{0}}}) $. Since $\theta_{1}=\theta_{1}' \cap (Q_{0}/\chi ^{\mathbf{Q_{0}}})$ and $\theta_{2}=\theta_{2}' \cap (Q_{0}/\tau ^{\mathbf{Q_{0}}})$, we have $\langle  x/ \chi ^{\mathbf{Q_{0}}}, y/ \chi ^{\mathbf{Q_{0}}} \rangle \in \theta_{1}$ and $\langle  x/ \tau ^{\mathbf{Q_{0}}}, y/ \tau ^{\mathbf{Q_{0}}} \rangle \in \theta_{2}$, it turns out that $\langle \langle x/ \chi ^{\mathbf{Q_{0}}},x/ \tau ^{\mathbf{Q_{0}}} \rangle , \\ \langle y/ \chi ^{\mathbf{Q_{0}}},y/ \tau ^{\mathbf{Q_{0}}} \rangle \rangle \in \theta_{1} \times \theta_{2} $, so $(\theta_{1}'\times \theta_{2}')\cap ((Q_{0}/ \chi ^{\mathbf{Q_{0}}})\times (Q_{0}/ \tau ^{\mathbf{Q_{0}}})) \subseteq \theta_{1} \times \theta_{2} $. Therefore, $\theta_{1} \times \theta_{2}=(\theta_{1}'\times \theta_{2}')\cap ((Q_{0}/ \chi ^{\mathbf{Q_{0}}}) \times (Q_{0}/ \tau ^{\mathbf{Q_{0}}}))$.
Thus the variety $\mathbb{QB}$ has the CEP.
\end{proof}

Below, we use an example to illustrate the process of congruence extension.

\begin{example}\label{Example E} Let $\mathbf{4}$  be defined in Example \ref{Example 4} and  $\mathbf{6} = \langle \{0,a,e,f,b,1\} ; \vee ,\wedge ,^\ast, 0, 1 \rangle$ be defined as follows\\
$$
\begin{array}{c|cccccc} \vee & 0 & a & e & f & b & 1 \\ \hline
0 & 0 & 0 & 0 & 1 & 1 & 1 \\a & 0 & 0 & 0 & 1 & 1 & 1 \\e & 0 & 0 & 0 & 1 & 1 & 1 \\f & 1 & 1 & 1 & 1 & 1 & 1 \\b & 1 & 1 & 1 & 1 & 1 & 1 \\1 & 1 & 1 & 1 & 1 & 1 & 1
\end{array}
\qquad
\begin{array}{c|cccccc} \wedge & 0 & a & e & f & b & 1 \\ \hline
0 & 0 & 0 & 0 & 0 & 0 & 0 \\a & 0 & 0 & 0 & 0 & 0 & 0 \\e & 0 & 0 & 0 & 0 & 0 & 0 \\f & 0 & 0 & 0 & 1 & 1 & 1 \\b & 0 & 0 & 0 & 1 & 1 & 1 \\1 & 0 & 0 & 0 & 1 & 1 & 1
\end{array}
\qquad
\begin{array}{c|c} x & x^{*} \\ \hline
0 & 1 \\a & b \\e & f \\f & e \\b & a \\1 & 0
\end{array}
$$

Then it is easy to see that $\mathbf{4}$ is a subalgebra of $\mathbf{6}$ and $\mathbf{6}$ is embedded into the direct product $\mathbf{2}\times\mathbf{F_{5}}$. Put $\theta=\nabla$ the largest congruence on $\mathbf{4}$. Then following Lemma $\ref{Lemma 6}$, we have that
$\theta_{1}=\{\langle 0/\chi^{\mathbf{4}},0/\chi^{\mathbf{4}}\rangle,\langle 0/\chi^{\mathbf{4}},$ $1/\chi^{\mathbf{4}}\rangle,\langle 1/\chi^{\mathbf{4}},0/\chi^{\mathbf{4}}\rangle,\langle 1/\chi^{\mathbf{4}},1/\chi^{\mathbf{4}}\rangle\}$ where $0/\chi^{\mathbf{4}}=\{0,a\}$ and $1/\chi^{\mathbf{4}}=\{b,1\}$, is a congruence on $\mathbf{4}/\chi^{\mathbf{4}}$, and $\theta_{2}=\{\langle 0/\tau^{\mathbf{4}},0/\tau^{\mathbf{4}}\rangle,\langle 0/\tau^{\mathbf{4}},a/\tau^{\mathbf{4}}\rangle,$ $\langle a/\tau^{\mathbf{4}},0/\tau^{\mathbf{4}}\rangle,\langle 0/\tau^{\mathbf{4}},b/\tau^{\mathbf{4}}\rangle,$ $\langle b/\tau^{\mathbf{4}},0/\tau^{\mathbf{4}}\rangle, $ $\langle a/\tau^{\mathbf{4}},a/\tau^{\mathbf{4}}\rangle,\langle b/\tau^{\mathbf{4}},b/\tau^{\mathbf{4}}\rangle,\langle a/\tau^{\mathbf{4}},b/\tau^{\mathbf{4}}\rangle,\langle b/\tau^{\mathbf{4}},a/\tau^{\mathbf{4}}\rangle\}$ where $0/\tau^{\mathbf{4}}=\{0,1\}$, $a/\tau^{\mathbf{4}}=\{a\}$, and $b/\tau^{\mathbf{4}}=\{b\}$,  is a congruence on $\mathbf{4}/\tau^{\mathbf{4}}$. According to Lemma \ref{Lemma 5} and Lemma \ref{Lem subalgebra}, we get that $\theta_{2}'=\theta_{2}\cup\{\langle e,e\rangle,\langle f,f\rangle\}$ is a congruence on $\mathbf{6}/\tau^{\mathbf{6}}$. Because $\mathbf{4}/\chi^{\mathbf{4}}$ and $\mathbf{6}/\chi^{\mathbf{6}}$ are isomorphic, we denote $\theta_{1}'=\{\langle 0/\chi^{\mathbf{6}},0/\chi^{\mathbf{6}}\rangle,$ $\langle 0/\chi^{\mathbf{6}},1/\chi^{\mathbf{6}}\rangle, \langle 1/\chi^{\mathbf{6}},0/\chi^{\mathbf{6}}\rangle,\langle 1/\chi^{\mathbf{6}},1/\chi^{\mathbf{6}}\rangle\}$ where $0/\chi^{\mathbf{6}}=\{0,a,e\}$ and $1/\chi^{\mathbf{6}}=\{b,f,1\}$ and then $\theta_{1}'$ is a congruence on $\mathbf{6}/\chi^{\mathbf{6}}$. Hence $\theta_{1}'\times \theta_{2}'$ is a congruence on $\mathbf{2}\times\mathbf{F_{5}}$ and $\theta_{1}\times \theta_{2}=(\theta_{1}'\times \theta_{2}')\cap (\mathbf{2}\times \mathbf{F_{3}})$.
\end{example}

Since the subalgebra $\mathscr{R}(\mathbf{Q})$ of a QB-algebra $\mathbf{Q}$ is a Boolean algebra and the variety $\mathbb{QB}$ has the CEP, in the following, we present the specific method for extending a congruence on  $\mathscr{R}(\mathbf{Q})$ to a congruence on $\mathbf{Q}$ in order to investigate more relationships between Boolean algebras and quasi-Boolean algebras. For the sake of clarity, we also discuss flat QB-algebras and non-flat QB-algebras, respectively.

\begin{lemma}\label{xy equ}
Let $\mathbf{Q}$ be a non-flat QB-algebra and $\theta$ be a congruence on $\mathscr{R}(\mathbf{Q})$. Then $\theta_{x,y}=\theta \cup \Delta \cup \{\langle x,y\rangle, \langle y,x\rangle, \langle x^{\ast},y^{\ast}\rangle, \langle y^{\ast},x^{\ast}\rangle \}$ where $x\in \mathscr{IR}(Q)$ and $y\in cl(x\vee x)\backslash\mathscr{R}(Q)$ is an equivalence relation on $\mathbf{Q}$.
\end{lemma}

\begin{proof}
Obviously, $\theta_{x,y}$ is reflexive and symmetric. For any $\langle a,b\rangle \in \theta_{x,y}$ and $\langle b,c\rangle \in \theta_{x,y}$, we distinguish several cases to discuss.

(1) Suppose that $\langle a,b\rangle \in \theta$. If $\langle b,c\rangle \in \theta$, since $\theta$ is a congruence on $\mathscr{R}(\mathbf{Q})$, we have $\langle a,c\rangle \in \theta$, so $\langle a,c\rangle \in \theta_{x,y}$. If $\langle b,c\rangle \in \Delta$, then we have $b=c$, it turns out that $\langle a,c\rangle \in \theta$, so $\langle a,c\rangle \in \theta_{x,y}$. If $\langle b,c\rangle\in \{\langle x,y\rangle, \langle y,x\rangle, \langle x^{\ast},y^{\ast}\rangle, \langle y^{\ast},x^{\ast}\rangle \}$, then we have $b\in \mathscr{R}(Q)$ and $b\in \mathscr{IR}(Q)$, this is impossible.

(2) Suppose that $\langle a,b\rangle \in \Delta$. Then we have $a=b$. If $\langle b,c\rangle\in \theta\cup \Delta$, then $\langle a,c\rangle\in\theta\cup \Delta\subseteq \theta_{x,y}$. If $\langle b,c\rangle\in \{\langle x,y\rangle, \langle y,x\rangle, \langle x^{\ast},y^{\ast}\rangle, \langle y^{\ast},x^{\ast}\rangle \}$, then we also have $b\in \mathscr{R}(Q)$ and $b\in \mathscr{IR}(Q)$, this is impossible.

(3) Suppose that $\langle a,b\rangle \in \{\langle x,y\rangle, \langle y,x\rangle, \langle x^{\ast},y^{\ast}\rangle, \langle y^{\ast},x^{\ast}\rangle \}$. If $\langle b,c\rangle\in \theta$, then we have $b\in \mathscr{IR}(Q)$ and $b\in \mathscr{R}(Q)$, this is impossible. If $\langle b,c\rangle \in \Delta$, then we have $b=c$ and then $\langle a,c\rangle \in \{\langle x,y\rangle, \langle y,x\rangle, \langle x^{\ast},\\ y^{\ast}\rangle, \langle y^{\ast},x^{\ast}\rangle \}$, so $\langle a,c\rangle \in \theta_{x,y}$. If $\langle b,c\rangle \in \{\langle x,y\rangle, \langle y,x\rangle, \langle x^{\ast},y^{\ast}\rangle, \langle y^{\ast},x^{\ast}\rangle \}$, we can distinguish several cases. (3.1) If $\langle a,b\rangle=\langle x,y\rangle$ and $\langle b,c\rangle=\langle x,y\rangle$, then we have $b=y=x$, it turns out that $\langle a,c\rangle=\langle x,x\rangle\in \Delta\subseteq \theta_{x,y}$. (3.2) If $\langle a,b\rangle=\langle x,y\rangle$ and $\langle b,c\rangle=\langle y,x\rangle$, then we have $\langle a,c\rangle=\langle x,x\rangle\in \Delta\subseteq \theta_{x,y}$. (3.3) If $\langle a,b\rangle=\langle x,y\rangle$ and $\langle b,c\rangle=\langle x^{\ast},y^{\ast}\rangle$, then we have $b=y=x^{\ast}$, it turns out that $x=(x^{\ast})^{\ast}=y^{\ast}$, so $\langle a,c\rangle=\langle x,y^{\ast}\rangle=\langle x,x\rangle\in \Delta\subseteq \theta_{x,y}$. (3.4) If $\langle a,b\rangle=\langle x,y\rangle$ and $\langle b,c\rangle=\langle y^{\ast},x^{\ast}\rangle$, then we have $b=y=y^{\ast}$, this is impossible by Lemma \ref{neq}. The rest cases can be proved similarly.

Thus $\theta_{x,y}$ is an equivalence relation on $\mathbf{Q}$.
\end{proof}

\begin{proposition}\label{xy con}
Let $\mathbf{Q}$ be a non-flat QB-algebra and $\theta$ be a congruence on $\mathscr{R}(\mathbf{Q})$. Then $\theta_{x,y}=\theta \cup \Delta \cup \{\langle x,y\rangle, \langle y,x\rangle, \langle x^{\ast},y^{\ast}\rangle, \langle y^{\ast},x^{\ast}\rangle \}$ where $x\in \mathscr{IR}(Q)$ and $y\in cl(x\vee x)\backslash\mathscr{R}(Q)$ is a congruence on $\mathbf{Q}$.
\end{proposition}

\begin{proof}
By Lemma \ref{xy equ}, we have that $\theta_{x,y}$ is an equivalence relation on $\mathbf{Q}$. For any $\langle a,b\rangle\in \theta_{x,y}$ and $\langle c,d\rangle\in \theta_{x,y}$, we distinguish several cases to discuss.

(1) Suppose that $\langle a,b\rangle\in \theta$. Since $\theta$ is a congruence on $\mathscr {R}({\mathbf{Q}})$, we have $\langle a^{\ast},b^{\ast}\rangle\in \theta\subseteq \theta_{x,y}$. Moreover, if $\langle c,d\rangle\in \theta$, then we have $\langle a\vee c,b\vee d\rangle\in \theta\subseteq \theta_{x,y}$ and $\langle a\wedge c,b\wedge d\rangle\in \theta\subseteq \theta_{x,y}$. If $\langle c,d\rangle\in \Delta$, then we have $c=d$. Since $a\vee c, b\vee c,a\wedge c,\mathrm{and}\; b\wedge c\in \mathscr{R}(Q)$, we have $\langle a\vee c,b\vee d\rangle=\langle a\vee c,b\vee c\rangle\in \theta\subseteq \theta_{x,y}$ and $\langle a\wedge c,b\wedge d\rangle=\langle a\wedge c,b\wedge c\rangle\in \theta\subseteq \theta_{x,y}$. If $\langle c,d\rangle\in \{\langle x,y\rangle, \langle y,x\rangle, \langle x^{\ast},y^{\ast}\rangle, \langle y^{\ast},x^{\ast}\rangle \}$, then we have $c\vee c=x\vee x=y\vee y=d\vee d$ or $c\vee c=x^{\ast}\vee x^{\ast}=y^{\ast}\vee y^{\ast}=d\vee d$, it turns out that $\langle a\vee c,b\vee d\rangle=\langle a\vee (c\vee c),b\vee (d\vee d)\rangle=\langle a\vee (c\vee c),b\vee (c\vee c)\rangle\in \theta\subseteq \theta_{x,y}$ and $\langle a\wedge c,b\wedge d\rangle=\langle a\wedge (c\wedge c),b\wedge (d\wedge d)\rangle=\langle a\wedge (c\wedge c),b\wedge (c\wedge c)\rangle\in \theta\subseteq \theta_{x,y}$.

(2) Suppose that $\langle a,b\rangle\in \Delta$. Then we have $a=b$ and then $\langle a^{\ast},b^{\ast}\rangle\in \Delta\subseteq \theta_{x,y}$. Moreover, if $\langle c,d\rangle\in \theta$, since $a\vee c, a\vee d,a\wedge c,\;\mathrm{and}\;a\wedge d\in \mathscr{R}(Q)$, we have $\langle a\vee c,b\vee d\rangle=\langle a\vee c,a\vee d\rangle\in \theta\subseteq \theta_{x,y}$ and $\langle a\wedge c,b\wedge d\rangle=\langle a\wedge c,a\wedge d\rangle\in \theta\subseteq \theta_{x,y}$. If $\langle c,d\rangle\in \Delta$, then we have $c=d$, it implies that $\langle a\vee c,b\vee d\rangle=\langle a\vee c,a\vee c\rangle\in \Delta\subseteq \theta_{x,y}$. If $\langle c,d\rangle\in \{\langle x,y\rangle, \langle y,x\rangle, \langle x^{\ast},y^{\ast}\rangle, \langle y^{\ast},x^{\ast}\rangle \}$, then we have $c\vee c=d\vee d$ and $c\wedge c=d\wedge d$, it turns out that $\langle a\vee c, b\vee d\rangle=\langle a\vee (c\vee c),b\vee (d\vee d)\rangle=\langle a\vee (c\vee c),a\vee (c\vee c)\rangle\in \Delta \subseteq \theta_{x,y}$ and $\langle a\wedge c, b\wedge d\rangle=\langle a\wedge (c\wedge c),b\wedge (d\wedge d)\rangle=\langle a\wedge (c\wedge c),a\wedge (c\wedge c)\rangle\in \Delta \subseteq \theta_{x,y}$.

(3) Suppose that $\langle a,b\rangle\in \{\langle x,y\rangle, \langle y,x\rangle, \langle x^{\ast},y^{\ast}\rangle, \langle y^{\ast},x^{\ast}\rangle \}$. Then we have $a\vee a=b\vee b$ and $a\wedge a=b\wedge b$. Obviously, $\langle a^{\ast},b^{\ast}\rangle\in \{\langle x,y\rangle, \langle y,x\rangle, \langle x^{\ast},y^{\ast}\rangle,  \langle y^{\ast},x^{\ast}\rangle \}\subseteq \theta_{x,y}$. Moreover, if $\langle c,d\rangle\in \theta$, then we have $\langle a\vee c,b\vee d\rangle=\langle (a\vee a)\vee c,(b\vee b)\vee d\rangle=\langle (a\vee a)\vee c,(a\vee a)\vee d\rangle=\langle a\vee c,a\vee d\rangle\in \theta\subseteq \theta_{x,y}$ and $\langle a\wedge c,b\wedge d\rangle=\langle (a\wedge a)\wedge c,(b\wedge b)\wedge d\rangle=\langle (a\wedge a)\wedge c,(a\wedge a)\wedge d\rangle=\langle a\wedge c,a\wedge d\rangle\in \theta\subseteq \theta_{x,y}$. If $\langle c,d\rangle\in \Delta$, then we have $c=d$, it turns out that $\langle a\vee c,b\vee d\rangle=\langle (a\vee a)\vee c,(b\vee b)\vee d\rangle=\langle (a\vee a)\vee c,(a\vee a)\vee c\rangle\in \Delta\subseteq \theta_{x,y}$ and $\langle a\wedge c,b\wedge d\rangle=\langle (a\wedge a)\wedge c,(a\wedge a)\wedge c\rangle\in \Delta\subseteq \theta_{x,y}$. If $\langle c,d\rangle\in \{\langle x,y\rangle, \langle y,x\rangle, \langle x^{\ast},y^{\ast}\rangle, \langle y^{\ast},x^{\ast}\rangle \}$, then we have $c\vee c=d\vee d$ and $c\wedge c=d\wedge d$, so $\langle a\vee c,b\vee d\rangle=\langle (a\vee a)\vee (c\vee c),(b\vee b)\vee (d\vee d)\rangle=\langle (a\vee a)\vee (c\vee c),(a\vee a)\vee (c\vee c)\rangle\in \Delta\subseteq \theta_{x,y}$ and $\langle a\wedge c,b\wedge d\rangle=\langle (a\wedge a)\wedge (c\wedge c),(a\wedge a)\wedge (c\wedge c)\rangle\in \Delta\subseteq \theta_{x,y}$.

Hence $\theta_{x,y}$ is a congruence on $\mathbf{Q}$.
\end{proof}

\begin{remark} Let $\mathbf{Q}$ be a non-flat QB-algebra. Then the congruence $\theta_{x,y}$ defined in Proposition \ref{xy con} is the least congruence on $\mathbf{Q}$
which contains $\theta$ and $\langle x,y\rangle$.
\end{remark}

\begin{corollary}
Let $\mathbf{Q}$ be a non-flat QB-algebra. Then

\emph{(1)} if $\theta$ is a congruence on $\mathscr{R}(\mathbf{Q})$. Then $\theta_{x,y}$ defined in Proposition \ref{xy con} is an atom in $\emph{Con}(\mathbf{Q})$, where $\emph{Con}(\mathbf{Q})$ is the lattice of all congruences on $\mathbf{Q}$.

\emph{(2)} for any $\rho\in \emph{Con}(\mathbf{Q})$, there exists a congruence $\theta=\rho\cap \mathscr{R}(Q)^{2}$ on $\mathscr{R}(\mathbf{Q})$ such that $\rho$ is generated by $\{\theta_{x,y}\,|\, x,y\in X\}$ for a suitable subset $X\subseteq \mathscr{IR}(Q)$.
\end{corollary}

Below we will see the case of flat QB-algebras.

\begin{lemma}\label{xy equ2}
Let $\mathbf{Q}$ be a flat QB-algebra. Then for any $x,y\in \mathscr{IR}(Q)$ and $x\neq y$,
\begin{eqnarray*}
\theta_{x,y}= \left\{
          \begin{array}{ll}
      \Delta \cup \{\langle x,y\rangle, \langle y,x\rangle \}, \;\mathrm{if}\;x=x^{\ast}\mathrm{and}\; y= y^{\ast};\\
      \Delta \cup \{\langle x,y\rangle, \langle y,x\rangle, \langle x^{\ast},y^{\ast}\rangle, \langle y^{\ast},x^{\ast}\rangle,\langle y,y^{\ast}\rangle,\langle y^{\ast},y\rangle \},\;\mathrm{if}\; x=x^{\ast}\; \mathrm{and}\; y\neq y^{\ast};\\
      \Delta \cup \{\langle x,y\rangle, \langle y,x\rangle, \langle x^{\ast},y^{\ast}\rangle, \langle y^{\ast},x^{\ast}\rangle,\langle x,x^{\ast}\rangle,\langle x^{\ast},x\rangle \},\;\mathrm{if}\; x\neq x^{\ast}\; \mathrm{and}\; y= y^{\ast};\\
      \Delta \cup \{\langle x,y\rangle, \langle y,x\rangle, \langle x^{\ast},y^{\ast}\rangle, \langle y^{\ast},x^{\ast}\rangle,\langle x,x^{\ast}\rangle,\langle x^{\ast},x\rangle,\langle y,y^{\ast}\rangle,\langle y^{\ast},y\rangle,\langle x,y^{\ast}\rangle,\\ \langle y^{\ast},x\rangle,\langle x^{\ast},y\rangle,\langle y,x^{\ast}\rangle \},\;\mathrm{if}\; x\neq x^{\ast}\; \mathrm{and}\; y\neq y^{\ast},
          \end{array}
        \right.
\end{eqnarray*}
is an equivalence relation on $\mathbf{Q}$.
\end{lemma}

\begin{proof} We only prove the case that $x\neq x^{\ast}$ and $ y\neq y^{\ast}$. Other cases can be derived from it.
Denote $\tilde{\theta}=\{\langle x,y\rangle,\langle y,x\rangle, \langle x^{\ast},y^{\ast}\rangle, \langle y^{\ast}, x^{\ast}\rangle,\langle x,x^{\ast}\rangle,\langle x^{\ast},x\rangle,\langle y,y^{\ast}\rangle,\langle y^{\ast},y\rangle,\langle x,y^{\ast}\rangle, \langle y^{\ast},$ $x\rangle,\langle x^{\ast},y\rangle,\langle y, x^{\ast}\rangle\}$.
Obviously, $\theta_{x,y}$ is reflexive and symmetric.
For any $\langle a,b\rangle\in \theta_{x,y}$ and $\langle b,c\rangle\in \theta_{x,y}$,
if $\langle a,b\rangle\in \Delta$ or $\langle b,c\rangle\in \Delta$, then we have $a=b$ or $b=c$, it turns out that $\langle a,c\rangle\in \Delta \cup \tilde{\theta}$.
If $\langle a,b\rangle\in\tilde{\theta}$ and $\langle b,c\rangle\in\tilde{\theta}$, then it is easily verified that $\langle a,c\rangle\in \Delta \cup \tilde{\theta}$.
Thus $\theta_{x,y}$ is an equivalence relation on $\mathbf{Q}$.
\end{proof}

\begin{proposition}\label{xy con2}
Let $\mathbf{Q}$ be a flat QB-algebra. Then for any $x,y\in \mathscr{IR}(Q)$ and $x\neq y$,
\begin{eqnarray*}
\theta_{x,y}= \left\{
          \begin{array}{ll}
      \Delta \cup \{\langle x,y\rangle, \langle y,x\rangle \}, \;\mathrm{if}\;x=x^{\ast}\mathrm{and}\; y= y^{\ast};\\
      \Delta \cup \{\langle x,y\rangle, \langle y,x\rangle, \langle x^{\ast},y^{\ast}\rangle, \langle y^{\ast},x^{\ast}\rangle,\langle y,y^{\ast}\rangle,\langle y^{\ast},y\rangle \},\;\mathrm{if}\; x=x^{\ast}\; \mathrm{and}\; y\neq y^{\ast};\\
      \Delta \cup \{\langle x,y\rangle, \langle y,x\rangle, \langle x^{\ast},y^{\ast}\rangle, \langle y^{\ast},x^{\ast}\rangle,\langle x,x^{\ast}\rangle,\langle x^{\ast},x\rangle \},\;\mathrm{if}\; x\neq x^{\ast}\; \mathrm{and}\; y= y^{\ast};\\
      \Delta \cup \{\langle x,y\rangle, \langle y,x\rangle, \langle x^{\ast},y^{\ast}\rangle, \langle y^{\ast},x^{\ast}\rangle,\langle x,x^{\ast}\rangle,\langle x^{\ast},x\rangle,\langle y,y^{\ast}\rangle,\langle y^{\ast},y\rangle,\langle x,y^{\ast}\rangle,\\ \langle y^{\ast},x\rangle,\langle x^{\ast},y\rangle,\langle y,x^{\ast}\rangle \},\;\mathrm{if}\; x\neq x^{\ast}\; \mathrm{and}\; y\neq y^{\ast},
          \end{array}
        \right.
\end{eqnarray*}
is a congruence on $\mathbf{Q}$.
\end{proposition}

\begin{proof}
By Lemma \ref{xy equ2}, we have that $\theta_{x,y}$ is an equivalence relation on $\mathbf{Q}$. For any $\langle a,b\rangle\in \theta_{x,y}$, we have $\langle a^{\ast},b^{\ast}\rangle\in \theta_{x,y}$ obviously. For any $\langle a,b\rangle\in \theta_{x,y}$ and $\langle c,d\rangle\in \theta_{x,y}$, since $\mathbf{Q}$ is a flat QB-algebra, we have $\langle a\vee c,b\vee d\rangle=\langle a\wedge c,b\wedge d\rangle=\langle 0,0\rangle\in \Delta \subseteq \theta_{x,y}$ by Lemma \ref{vee 0}. Thus $\theta_{x,y}$ is a congruence on $\mathbf{Q}$.
\end{proof}

\begin{remark} Let $\mathbf{Q}$ be a flat QB-algebra. Then the congruence $\theta_{x,y}$ defined in Proposition \ref{xy con2} is the least congruence on $\mathbf{Q}$
which contains $\Delta$ and $\langle x,y\rangle$.
\end{remark}

Suppose that $\mathbf{Q}$ is a flat QB-algebra. Then, by Corollary \ref{Cor RQ}, we have $\mathscr{R}(Q)=\{0\}$ and hence $\{\langle 0,0\rangle\}$ is the unique congruence on
$\mathscr{R}({\mathbf{Q}})$. Therefore,  the extended congruence from Proposition \ref{xy con2} can be decomposed into two parts, one on $\mathscr{R}({\mathbf{Q}})$ and the other  on
$\mathscr{IR}({Q})$. The result is as follows.

\begin{proposition}\label{p5.3}
Let $\mathbf{Q}$ be a flat QB-algebra and $\Delta_{\mathscr{R}(Q)}$ be its only congruence on $\mathscr{R}({\mathbf{Q}})$. Then $\theta=\Delta_{\mathscr{R}(Q)}\cup \theta_{\mathscr{IR}(Q)}$ where $\theta_{\mathscr{IR}(Q)}$ is an equivalence relation on $\mathscr{IR}({Q})$ with $\theta_{\mathscr{IR}(Q)}^{\ast}=\theta_{\mathscr{IR}(Q)}$, is a congruence on $\mathbf{Q}$.
\end{proposition}

\begin{proof}
Obviously, $\theta$ is reflexive and symmetric. For any $\langle a,b\rangle\in \theta$ and $\langle b,c\rangle\in \theta$, if $\langle a,b\rangle\in \Delta_{\mathscr{R}(Q)}$ and $\langle b,c\rangle\in \Delta_{\mathscr{R}(Q)}$, then $a=b=c=0$, we have $\langle a,c\rangle\in \Delta_{\mathscr{R}(Q)}\subseteq \theta$. If $\langle a,b\rangle\in \theta_{\mathscr{IR}(Q)}$ and $\langle b,c\rangle\in \theta_{\mathscr{IR}(Q)}$, since $\theta_{\mathscr{IR}(Q)}$ is an equivalence relation on $\mathscr{IR}({\mathbf{Q}})$, we have $\langle a,c\rangle\in \theta_{\mathscr{IR}(Q)}\subseteq \theta$. If $\langle a,b\rangle\in \Delta_{\mathscr{R}(Q)}$ and $\langle b,c\rangle\in \theta_{\mathscr{IR}(Q)}$, then we have $b\in \mathscr{R}(Q)$ and $b\in \mathscr{IR}(Q)$, this is impossible. If $\langle a,b\rangle\in \theta_{\mathscr{IR}(Q)}$ and $\langle b,c\rangle\in \Delta_{\mathscr{R}(Q)}$, this is also impossible. Hence $\theta$ is an equivalence relation on $\mathbf{Q}$. For any $\langle a,b\rangle\in \theta$ and $\langle c,d\rangle\in \theta$, since $\mathbf{Q}$ is a flat QB-algebra, we have $\langle a\vee c,b\vee d\rangle=\langle a\wedge c,b\wedge d\rangle=\langle 0,0\rangle\in \Delta_{\mathscr{R}(Q)}$ by Lemma \ref{vee 0}. Moreover, if $\langle a,b\rangle\in \Delta_{\mathscr{R}(Q)}$, then $a=b=0$, we have $\langle a^{\ast},b^{\ast}\rangle=\langle 0,0\rangle\in \Delta_{\mathscr{R}(Q)}\subseteq \theta$. If $\langle a,b\rangle\in \theta_{\mathscr{IR}(Q)}$, since $\theta_{\mathscr{IR}(Q)}^{\ast}=\theta_{\mathscr{IR}(Q)}$, we have $\langle a^{\ast},b^{\ast}\rangle\in \theta_{\mathscr{IR}(Q)}\subseteq \theta$. Thus $\theta$ is a congruence on $\mathbf{Q}$.
\end{proof}

Through Proposition \ref{p5.3}, we can see that for a flat QB-algebra $\mathbf{Q}$, the congruence on $\mathscr{R}(\mathbf{Q})$ combined with an equivalence relation on $\mathscr{IR}(Q)$ satisfying a certain condition yields a congruence on $\mathbf{Q}$. Now, a natural question arises: does this approach also apply to non-flat QB-algebras? Below we will further explore this question and aim to fully characterize congruences on a non-flat QB-algebra $\mathbf{Q}$ using the congruences on its subalgebra $\mathscr{R}(\mathbf{Q})$.

\begin{proposition}\label{break}
Let $\mathbf{Q}$ be a non-flat QB-algebra and $\theta_{\mathscr{R}(\mathbf{Q})}$ be a congruence on $\mathscr{R}(\mathbf{Q})$. Assume that

\emph{(C1)} $\theta_{\mathscr{IR}(Q)}$ is an equivalence relation on $\mathscr{IR}(Q)$ with $\theta_{\mathscr{IR}(Q)}^{\ast}=\theta_{\mathscr{IR}(Q)}$ and $\theta_{\mathscr{IR}(Q)} \subseteq \bigcup_{k\in \mathscr{R}({Q})}(cl(k/\theta_{\mathscr{R}(Q)})\times (cl(k/\theta_{\mathscr{R}(Q)})$;

\emph{(C2)} there exists an injective mapping $f: \mathscr{R}(Q)/\theta_{\mathscr{R}(Q)}\rightarrow \mathscr{IR}(Q)/\theta_{\mathscr{IR}(Q)}$ with $f(x/\theta_{\mathscr{R}(Q)})=y/\theta_{\mathscr{IR}(Q)}$ where $y\in cl(x)\cap \mathscr{IR} (Q)$ and $f$ preserves the operation $^{\ast}$;

\emph{(C3)} define a relation
\begin{eqnarray*}
&&\theta_{\mathscr{R}(Q)\times \mathscr{IR}(Q)}  \\
  &=&(\bigcup_{x\in \mathscr{R}(Q)}(x/\theta_{\mathscr{R}(Q)}\times f(x/\theta_{\mathscr{R}(Q)})))\cup (\bigcup_{x\in \mathscr{R}(Q)}(f(x/\theta_{\mathscr{R}(Q)}) \times x/\theta_{\mathscr{R}(Q)} ))  \\
  &=& (\bigcup_{x\in \mathscr{R}(Q)}(x/\theta_{\mathscr{R}(Q)}\times (y/\theta_{\mathscr{IR}(Q)})))\cup (\bigcup_{x\in \mathscr{R}(Q)}(y/\theta_{\mathscr{IR}(Q)} \times x/\theta_{\mathscr{R}(Q)})).
\end{eqnarray*}

Then $\theta=\theta_{\mathscr{R}(Q)}\bigcup \theta_{\mathscr{IR}(Q)} \bigcup \theta_{\mathscr{R}(Q)\times \mathscr{IR}(Q)}$ is a congruence on $\mathbf{Q}$.
\end{proposition}

To complete the proof of Proposition \ref{break},  we first prove that $\theta$ is an equivalence relation.

\begin{lemma}\label{l5.8}
$\theta=\theta_{\mathscr{R}(Q)}\bigcup \theta_{\mathscr{IR}(Q)} \bigcup \theta_{\mathscr{R}(Q)\times \mathscr{IR}(Q)}$ is an equivalence relation on $\mathbf{Q}$
\end{lemma}

\begin{proof}
For any $a\in Q$, we have $a\in \mathscr{R}(Q)$ or $a\in \mathscr{IR}(Q)$. If $a\in \mathscr{R}(Q)$, then we have $\langle a, a\rangle\in \theta_{\mathscr{R}(Q)} \subseteq \theta$. If $a\in \mathscr{IR}(Q)$, then we have $\langle a, a\rangle\in \theta_{\mathscr{IR}(Q)} \subseteq \theta$, so $\theta$ is reflexive. Note that $\theta_{\mathscr{R}(Q)}, \theta_{\mathscr{IR}(Q)}$, and $\theta_{\mathscr{R}(Q)\times \mathscr{IR}(Q)}$ are symmetric, so $\theta$ is symmetric. For any $\langle a,b\rangle \in \theta$ and $\langle b,c\rangle \in \theta$, we can distinguish several cases to discuss.

(1) Suppose that $\langle a,b \rangle\in \theta_{\mathscr{R}(Q)}$.  If $\langle b,c \rangle\in \theta_{\mathscr{R}(Q)}$, since $\theta_{\mathscr{R}(Q)}$ is a congruence on $\mathscr{R}(\mathbf{Q})$, we have $\langle a,c \rangle\in \theta_{\mathscr{R}(Q)}\subseteq\theta$.  If $\langle b,c \rangle\in \theta_{\mathscr{IR}(Q)}$, then we have $b\in \mathscr{R}(Q)$ and $b\in \mathscr{IR}(Q)$, this is impossible.  If $\langle b,c \rangle\in \theta_{\mathscr{R}(Q)\times \mathscr{IR}(Q)}$, then we have $b\in f(x/\theta_{\mathscr{R}(Q)})$ or $c\in f(x/\theta_{\mathscr{R}(Q)})$. Notice that $\langle a,b\rangle\in \theta_{\mathscr{R}(Q)}$, the former is impossible.
So we have $\langle b,c \rangle\in x/\theta_{\mathscr{R}(Q)}\times f(x/\theta_{\mathscr{R}(Q)})$ for some $x\in \mathscr{R}(Q)$ and then $b/\theta_{\mathscr{R}(Q)}=x/\theta_{\mathscr{R}(Q)}$ and $c/\theta_{\mathscr{IR}(Q)}=f(x/\theta_{\mathscr{R}(Q)})$.
Since $\langle a,b \rangle\in \theta_{\mathscr{R}(Q)}$, we have $a/\theta_{\mathscr{R}(Q)}=x/\theta_{\mathscr{R}(Q)}$, it turns out that $f(a/\theta_{\mathscr{R}(Q)})=f(x/\theta_{\mathscr{R}(Q)})=c/\theta_{\mathscr{IR}(Q)}$, so $\langle a,c\rangle\in \theta_{\mathscr{R}(Q)\times \mathscr{IR}(Q)}\subseteq\theta$.

(2) Suppose that $\langle a,b \rangle\in \theta_{\mathscr{IR}(Q)}$.  If $\langle b,c \rangle\in \theta_{\mathscr{R}(Q)}$, then we have $b\in \mathscr{IR}(Q)$ and $b\in \mathscr{R}(Q)$, this is impossible.  If $\langle b,c \rangle\in \theta_{\mathscr{IR}(Q)}$, since $\theta_{\mathscr{IR}(Q)}$ is an equivalence relation on $\mathscr{IR}(Q)$, we have $\langle a,c \rangle\in \theta_{\mathscr{IR}(Q)}\subseteq\theta$.  If $\langle b,c \rangle\in \theta_{\mathscr{R}(Q)\times \mathscr{IR}(Q)}$, then we have $b\in x/\theta_{\mathscr{R}(Q)}$ or $c\in x/\theta_{\mathscr{R}(Q)}$. Notice that $\langle a,b\rangle\in \theta_{\mathscr{IR}(Q)}$, the former is impossible.
So we have $\langle b,c \rangle\in f(x/\theta_{\mathscr{R}(Q)})\times x/\theta_{\mathscr{R}(Q)}$ for some $x\in \mathscr{R}(Q)$ and then $b/\theta_{\mathscr{IR}(Q)}=f(x/\theta_{\mathscr{R}(Q)})$ and $c/\theta_{\mathscr{R}(Q)}=x/\theta_{\mathscr{R}(Q)}$.
Since $\langle a,b \rangle\in \theta_{\mathscr{IR}(Q)}$, we have $a/\theta_{\mathscr{IR}(Q)}=b/\theta_{\mathscr{IR}(Q)}$, it turns out that $f(c/\theta_{\mathscr{R}(Q)})=f(x/\theta_{\mathscr{R}(Q)})=b/\theta_{\mathscr{IR}(Q)}=a/\theta_{\mathscr{IR}(Q)}$, so $\langle c,a\rangle\in \theta_{\mathscr{R}(Q)\times \mathscr{IR}(Q)}$. Since $\theta_{\mathscr{R}(Q)\times \mathscr{IR}(Q)}$ is symmetric, we have $\langle a,c\rangle\in \theta_{\mathscr{R}(Q)\times \mathscr{IR}(Q)}\subseteq\theta$.

(3) Suppose that  $\langle a,b \rangle\in\theta_{\mathscr{R}(Q)\times \mathscr{IR}(Q)}$.
If $\langle b,c \rangle\in \theta_{\mathscr{R}(Q)}$ or $\theta_{\mathscr{IR}(Q)}$, then $\langle a,c \rangle\in\theta_{\mathscr{R}(Q)\times \mathscr{IR}(Q)}$ can be proved similarly as above.
If $\langle b,c \rangle\in\theta_{\mathscr{R}(Q)\times \mathscr{IR}(Q)}$, then we have that $\langle a,b\rangle\in x/\theta_{\mathscr{R}(Q)}\times f(x/\theta_{\mathscr{R}(Q)})$
and $\langle b,c\rangle\in f(x'/\theta_{\mathscr{R}(Q)})\times x'/\theta_{\mathscr{R}(Q)}$ or $\langle a,b\rangle\in f(x/\theta_{\mathscr{R}(Q)})\times x/\theta_{\mathscr{R}(Q)}$
and $\langle b,c\rangle\in x'/\theta_{\mathscr{R}(Q)}\times f(x'/\theta_{\mathscr{R}(Q)})$ for some $x,x'\in \mathscr{R}(Q)$. If the former holds, then $a/\theta_{\mathscr{R}(Q)}=x/\theta_{\mathscr{R}(Q)}$, $b/\theta_{\mathscr{IR}(Q)}=f(x/\theta_{\mathscr{R}(Q)})$ and $b/\theta_{\mathscr{IR}(Q)}=f(x'/\theta_{\mathscr{R}(Q)})$,
$c/\theta_{\mathscr{R}(Q)}=x'/\theta_{\mathscr{R}(Q)}$, it turns out that $f(x/\theta_{\mathscr{R}(Q)})=f(x'/\theta_{\mathscr{R}(Q)})$. Since $f$ is injective, we have $x/\theta_{\mathscr{R}(Q)}=x'/\theta_{\mathscr{R}(Q)}$, so $a/\theta_{\mathscr{R}(Q)}=c/\theta_{\mathscr{R}(Q)}$ and then $\langle a,c\rangle\in \theta_{\mathscr{R}(Q)}\subseteq\theta$. If the latter holds, then $a/\theta_{\mathscr{IR}(Q)}=f(x/\theta_{\mathscr{R}(Q)})$, $b/\theta_{\mathscr{R}(Q)}=x/\theta_{\mathscr{R}(Q)}$ and $b/\theta_{\mathscr{R}(Q)}=x'/\theta_{\mathscr{R}(Q)}$,
$c/\theta_{\mathscr{IR}(Q)}=f(x'/\theta_{\mathscr{R}(Q)})$ for some $x, x' \in \mathscr{R}(Q)$, it turns out that $x/\theta_{\mathscr{R}(Q)}=x'/\theta_{\mathscr{R}(Q)}$, so $a/\theta_{\mathscr{IR}(Q)}=f(x/\theta_{\mathscr{R}(Q)})=f(x'/\theta_{\mathscr{R}(Q)})=c/\theta_{\mathscr{IR}(Q)}$ and then
$\langle a,c\rangle\in \theta_{\mathscr{IR}(Q)}\subseteq \theta$.

Hence $\theta$ is an equivalence relation on $\mathbf{Q}$.
\end{proof}

\noindent $\mathbf{Proof of Proposition}$ \ref{break}
\begin{proof} By Lemma \ref{l5.8}, we have that $\theta$ is an equivalence relation on $\mathbf{Q}$. To finish the proof, we still need to show that $\theta$ satisfies the compatibility property. For any $\langle a,b\rangle\in \theta$ and $\langle c,d\rangle\in \theta$, we can distinguish several cases to discuss.

(1) Suppose that $\langle a,b\rangle\in \theta_{\mathscr{R}(Q)}$. Since $\theta_{\mathscr{R}(Q)}$ is a congruence on $\mathscr{R}(\mathbf{Q})$, we have $\langle a^{\ast},b^{\ast}\rangle\in \theta_{\mathscr{R}(Q)}\subseteq \theta$. Now we consider the compatibility property with binary operations.
If $\langle c,d\rangle\in \theta_{\mathscr{R}(Q)}$, since $\theta_{\mathscr{R}(Q)}$ is a congruence on $\mathscr{R}(\mathbf{Q})$, we have $\langle a\vee b,c \vee d\rangle\in \theta_{\mathscr{R}(Q)}\subseteq \theta$ and $\langle a\wedge b,c \wedge d\rangle\in \theta_{\mathscr{R}(Q)}\subseteq \theta$.  If $\langle c,d\rangle\in \theta_{\mathscr{IR}(Q)}$, then there exists $k\in \mathscr{R}(Q)$ such that $\langle c,d\rangle\in cl(k/\theta_{\mathscr{R}(Q)})\times cl(k/\theta_{\mathscr{R}(Q)})$ by (C1), it turns out that there exist $k_{c},k_{d}\in k/\theta_{\mathscr{R}(Q)}$ such that $c\vee c=k_{c}\vee k_{c}=k_{c}=k_{c}\wedge k_{c}=c\wedge c$ and $d\vee d=k_{d}\vee k_{d}=k_{d}=k_{d}\wedge k_{d}=d\wedge d$, so $\langle a\vee c,b \vee d\rangle=\langle a\vee (c\vee c),b \vee (d\vee d)\rangle=\langle a\vee k_{c},b \vee k_{d}\rangle\in \theta_{\mathscr{R}(Q)}\subseteq \theta$ and $\langle a\wedge c,b \wedge d\rangle=\langle a\wedge (c\wedge c),b \wedge (d\wedge d)\rangle=\langle a\wedge k_{c},b \wedge k_{d}\rangle\in \theta_{\mathscr{R}(Q)}\subseteq \theta$. If $\langle c,d\rangle\in \theta_{\mathscr{R}(Q)\times \mathscr{IR}(Q)}$, then $\langle c,d\rangle\in x/\theta_{\mathscr{R}(Q)}\times f(x/\theta_{\mathscr{R}(Q)})$ or
$\langle c,d\rangle\in f(x/\theta_{\mathscr{R}(Q)})\times x/\theta_{\mathscr{R}(Q)}$ for some $x\in \mathscr{R}(Q)$. If the former holds,
then we have $c/\theta_{\mathscr{R}(Q)}=x/\theta_{\mathscr{R}(Q)}$ and $d/\theta_{\mathscr{IR}(Q)}=f(x/\theta_{\mathscr{R}(Q)})$.
By the definition of $\emph{f}$, we have $d\in cl(c/\theta_{\mathscr{R}(Q)})\cap \mathscr{IR}(Q)$, it turns out that there exists $t\in \mathscr{R}(Q)$ such that $t/\theta_{\mathscr{R}(Q)}= c/\theta_{\mathscr{R}(Q)}$ and $d\vee d=t\vee t=t$, so $\langle a\vee c,b\vee d\rangle=\langle a\vee c,b\vee (d\vee d)\rangle=\langle a\vee c,b\vee t\rangle\in \theta_{\mathscr{R}(Q)}\subseteq \theta$ and $\langle a\wedge c, b \wedge d\rangle=\langle a\wedge c, b\wedge (d\wedge d)\rangle=\langle a\wedge c, b\wedge t\rangle \in \theta_{\mathscr{R}(Q)}\subseteq \theta$.
If the latter holds, then $\langle a\vee c,b\vee d\rangle\subseteq \theta$ and $\langle a\wedge c,b\wedge d\rangle\subseteq \theta$ can be proved similarly.

(2) Suppose that $\langle a,b\rangle\in \theta_{\mathscr{IR}(Q)}$. Since $\theta_{\mathscr{IR}(Q)}^{\ast}=\theta_{\mathscr{IR}(Q)}$, we have $\langle a^{\ast},b^{\ast}\rangle\in \theta_{\mathscr{IR}(Q)}\subseteq \theta$. Moreover, since $\langle a,b\rangle\in \theta_{\mathscr{IR}(Q)}$, we have $k\in \mathscr{R}(Q)$ such that $\langle a,b\rangle\in cl(k/\theta_{\mathscr{R}(Q)})\times cl(k/\theta_{\mathscr{R}(Q)})$ by (C1), it turns out that there exist $k_{a},k_{b}\in k/\theta_{\mathscr{R}(Q)}$ such that $a\vee a=k_{a}\vee k_{a}=k_{a}=k_{a}\wedge k_{a}=a\wedge a$ and $b\vee b=k_{b}\vee k_{b}=k_{b}=k_{b}\wedge k_{b}=b\wedge b$. Now we consider the compatibility property with binary operations. If $\langle c,d\rangle\in \theta_{\mathscr{R}(Q)}$, then we have $\langle a\vee c,b \vee d\rangle=\langle (a\vee a)\vee c,(b \vee b)\vee d\rangle=\langle k_{a}\vee c ,k_{b}\vee d \rangle\in \theta_{\mathscr{R}(Q)}\subseteq \theta$ and $\langle a\wedge c,b \wedge d\rangle=\langle (a\wedge a)\wedge c,(b \wedge b)\wedge d\rangle=\langle k_{a}\wedge c,k_{b} \wedge d\rangle\in \theta_{\mathscr{R}(Q)}\subseteq \theta$. If $\langle c,d\rangle\in \theta_{\mathscr{IR}(Q)}$, then there exists $w\in \mathscr{R}(Q)$ such that $\langle c,d\rangle\in cl(w/\theta_{\mathscr{R}(Q)})\times cl(w/\theta_{\mathscr{R}(Q)})$ by (C1), it turns out that there exist $w_{c},w_{d}\in w/\theta_{\mathscr{R}(Q)}$ such that $c\vee c=w_{c}\vee w_{c}=w_{c}=w_{c}\wedge w_{c}=c\wedge c$ and $d\vee d=w_{d}\vee w_{d}=w_{d}=w_{d}\wedge w_{d}=d\wedge d$, so $\langle a\vee c,b\vee d\rangle=\langle (a\vee a)\vee (c\vee c),(b\vee b)\vee (d\vee d)\rangle=\langle k_{a}\vee w_{c},k_{b}\vee w_{d}\rangle\in \theta_{\mathscr{R}(Q)}\subseteq \theta$ and $\langle a\wedge c,b\wedge d\rangle=\langle (a\wedge a)\wedge (c\wedge c),(b\wedge b)\wedge (d\wedge d)\rangle=\langle k_{a}\wedge w_{c},k_{b}\wedge w_{d}\rangle\in \theta_{\mathscr{R}(Q)}\subseteq \theta$.
If $\langle c,d\rangle\in \theta_{\mathscr{R}(Q)\times \mathscr{IR}(Q)}$, then $\langle c,d\rangle\in x/\theta_{\mathscr{R}(Q)}\times f(x/\theta_{\mathscr{R}(Q)})$ or
$\langle c,d\rangle\in f(x/\theta_{\mathscr{R}(Q)})\times x/\theta_{\mathscr{R}(Q)}$ for some $x\in \mathscr{R}(Q)$. If the former holds,
then we have $c/\theta_{\mathscr{R}(Q)}=x/\theta_{\mathscr{R}(Q)}$ and $d/\theta_{\mathscr{IR}(Q)}=f(x/\theta_{\mathscr{R}(Q)})$.
By the definition of $\emph{f}$, we have $d\in cl(c/\theta_{\mathscr{R}(Q)})\cap \mathscr{IR}(Q)$, it turns out that there exists $t\in \mathscr{R}(Q)$ such that $t/\theta_{\mathscr{R}(Q)}= c/\theta_{\mathscr{R}(Q)}$ and $d\vee d=t\vee t=t$, so $\langle a\vee c,b\vee d\rangle=\langle (a\vee a)\vee c,(b\vee b)\vee (d\vee d)\rangle=\langle k_{a}\vee c,k_{b}\vee t\rangle\in \theta_{\mathscr{R}(Q)}\subseteq \theta$ and $\langle a\wedge c, b \wedge d\rangle=\langle (a\wedge a)\wedge c, (b\wedge b)\wedge (d\wedge d)\rangle=\langle k_{a}\wedge c, k_{b}\wedge t\rangle \in \theta_{\mathscr{R}(Q)}\subseteq \theta$.
If the latter holds, then $\langle a\vee c,b\vee d\rangle\subseteq \theta$ and $\langle a\wedge c,b\wedge d\rangle\subseteq \theta$ can be proved similarly.

(3) Suppose that $\langle a,b\rangle\in \theta_{\mathscr{R}(Q)\times \mathscr{IR}(Q)}$. Then $\langle a,b\rangle\in x/\theta_{\mathscr{R}(Q)}\times f(x/\theta_{\mathscr{R}(Q)})$ or
$\langle a,b\rangle\in f(x/\theta_{\mathscr{R}(Q)})\times x/\theta_{\mathscr{R}(Q)}$ for some $x\in \mathscr{R}(Q)$. If the former holds, then $a/\theta_{\mathscr{R}(Q)}=x/\theta_{\mathscr{R}(Q)}$ and $b/\theta_{\mathscr{IR}(Q)}=f(x/\theta_{\mathscr{R}(Q)})$ for some $x\in \mathscr{R}(Q)$,
it follows that $f(x^{\ast}/\theta_{\mathscr{R}(Q)})=f(x/\theta_{\mathscr{R}(Q)})^{\ast}=(b/\theta_{\mathscr{IR}(Q)})^{\ast}$ by (C2). Since $\theta_{\mathscr{IR}(Q)}^{\ast}=\theta_{\mathscr{IR}(Q)}$ by (C1), we have $(b/\theta_{\mathscr{IR}(Q)})^{\ast}$ $=b^{\ast}/\theta_{\mathscr{IR}(Q)}$,
it turns out that $a^{\ast}/\theta_{\mathscr{R}(Q)}=x^{\ast}/\theta_{\mathscr{R}(Q)}$ and $f(x^{\ast}/\theta_{\mathscr{R}(Q)})=b^{\ast}/\theta_{\mathscr{IR}(Q)}$, so $\langle a^{\ast},b^{\ast}\rangle\in \theta_{\mathscr{R}(Q)\times \mathscr{IR}(Q)}\subseteq \theta$. If the latter holds, then $\langle a^{\ast},b^{\ast}\rangle\in \theta$ can be proved similarly. Now we see the compatibility property with binary operations. If $\langle c,d\rangle\in \theta_{\mathscr{R}(Q)}$ or $\langle c,d\rangle\in \theta_{\mathscr{IR}(Q)}$, then the proof is similar to the cases (1) and (2).
If $\langle c,d\rangle\in \theta_{\mathscr{R}(Q)\times \mathscr{IR}(Q)}$, then $\langle c,d\rangle\in x'/\theta_{\mathscr{R}(Q)}\times f(x'/\theta_{\mathscr{R}(Q)})$ or
$\langle c,d\rangle\in f(x'/\theta_{\mathscr{R}(Q)})\times x'/\theta_{\mathscr{R}(Q)}$ for some $x'\in \mathscr{R}(Q)$. Without loss of generality, we assume that $\langle a,b\rangle\in x/\theta_{\mathscr{R}(Q)}\times f(x/\theta_{\mathscr{R}(Q)})$ for some  $x\in \mathscr{R}(Q)$. Then $a/\theta_{\mathscr{R}(Q)}=x/\theta_{\mathscr{R}(Q)}$, $b/\theta_{\mathscr{IR}(Q)}=f(x/\theta_{\mathscr{R}(Q)})$ and $c/\theta_{\mathscr{R}(Q)}=x'/\theta_{\mathscr{R}(Q)}$, $d/\theta_{\mathscr{IR}(Q)}=f(x'/\theta_{\mathscr{R}(Q)})$.
By the definition of $\emph{f}$, we have $b\in cl(a/\theta_{\mathscr{R}(Q)})\cap \mathscr{IR}(Q)$ and $d\in cl(c/\theta_{\mathscr{R}(Q)})\cap \mathscr{IR}(Q)$, it turns out that  there exist $s,t\in \mathscr{R}(Q)$ such that $s/\theta_{\mathscr{R}(Q)}= a/\theta_{\mathscr{R}(Q)}$ and $t/\theta_{\mathscr{R}(Q)}= c/\theta_{\mathscr{R}(Q)}$. Also, we have $b\vee b=s\vee s=s=s\wedge s=b\wedge b$  and $d\vee d=t\vee t=t=t\wedge t=d\wedge d$,
so $\langle a\vee c,b \vee d\rangle=\langle a\vee c,(b\vee b) \vee (d\vee d)\rangle=\langle a\vee c,s \vee t\rangle\in \theta_{\mathscr{R}(Q)}\subseteq \theta$ and $\langle a\wedge c,b \wedge d\rangle=\langle a\wedge c,(b\wedge b)\wedge (d\wedge d)\rangle=\langle a\wedge c,s \wedge t\rangle\in \theta_{\mathscr{R}(Q)}\subseteq \theta$. The rest can be proved similarly.

Hence $\theta$ is a congruence on $\mathbf{Q}$.
\end{proof}

We can generalize Proposition \ref{break} to the general case. Indeed, we  can replace the set $\mathscr{R}(Q)/\theta_{\mathscr{R}(Q)}$ in condition (C2) with an arbitrary subset $X$ satisfying $X\subseteq \mathscr{R}(Q)/\theta_{\mathscr{R}(Q)}$. The proof is similar and we omit it here.

\begin{proposition}\label{break2}
Let $\mathbf{Q}$ be a non-flat QB-algebra and $\theta_{\mathscr{R}(\mathbf{Q})}$ be a congruence on $\mathscr{R}(\mathbf{Q})$. Assume that

\emph{(C1)} $\theta_{\mathscr{IR}(Q)}$ is an equivalence relation on $\mathscr{IR}(Q)$ with $\theta_{\mathscr{IR}(Q)}^{\ast}=\theta_{\mathscr{IR}(Q)}$ and $\theta_{\mathscr{IR}(Q)} \subseteq \bigcup_{k\in \mathscr{R}({Q})}(cl(k/\theta_{\mathscr{R}(Q)})\times (cl(k/\theta_{\mathscr{R}(Q)})$;

\emph{(C2)} there exists an injective mapping $f: X \rightarrow \mathscr{IR}(Q)/\theta_{\mathscr{IR}(Q)}$ by $f(x/\theta_{\mathscr{R}(Q)})=y/\theta_{\mathscr{IR}(Q)}$ where
 $X^{\ast}=X\subseteq \mathscr{R}(Q)/\theta_{\mathscr{R}(Q)}$,  $y\in cl(x)\cap \mathscr{IR} (Q)$ and $f$ preserves the operation $^{\ast}$;

\emph{(C3)} define a relation
\begin{eqnarray*}
&&\theta_{\mathscr{R}(Q)\times \mathscr{IR}(Q)}  \\
  &=&(\bigcup_{x/\theta_{\mathscr{R}(Q)}\in X}(x/\theta_{\mathscr{R}(Q)}\times f(x/\theta_{\mathscr{R}(Q)})))\cup (\bigcup_{x/\theta_{\mathscr{R}(Q)}\in X}(f(x/\theta_{\mathscr{R}(Q)}) \times x/\theta_{\mathscr{R}(Q)} ))  \\
  &=& (\bigcup_{x/\theta_{\mathscr{R}(Q)}\in X}(x/\theta_{\mathscr{R}(Q)}\times (y/\theta_{\mathscr{IR}(Q)})))\cup (\bigcup_{x/\theta_{\mathscr{R}(Q)}\in X}(y/\theta_{\mathscr{IR}(Q)} \times x/\theta_{\mathscr{R}(Q)})).
\end{eqnarray*}

Then $\theta=\theta_{\mathscr{R}(Q)}\bigcup \theta_{\mathscr{IR}(Q)} \bigcup \theta_{\mathscr{R}(Q)\times \mathscr{IR}(Q)}$ is a congruence on $\mathbf{Q}$.
\end{proposition}

\begin{proposition}\label{theta123}
Let $\mathbf{Q}$ be a non-flat QB-algebra and $\theta$ be a congruence on $\mathbf{Q}$. Then there exist $\theta_{\mathscr{R}(Q)}$, $\theta_{\mathscr{IR}(Q)}$, and an injective mapping $f: X\rightarrow \mathscr{IR}(Q)/\theta_{\mathscr{IR}(Q)}$ which satisfy the conditions in Proposition \ref{break2}.
\end{proposition}

\begin{proof}
Suppose that $\theta$ is a congruence on the non-flat QB-algebra $\mathbf{Q}$. Consider $\theta_{\mathscr{R}(Q)}=\theta \cap \mathscr{R}(Q)^{2}$.  Then it is easy to see that $\theta_{\mathscr{R}(Q)}$ is a congruence on $\mathscr{R}(\mathbf{Q})$.
Put $\theta_{\mathscr{IR}(Q)}=\theta \cap \mathscr{IR}(Q)^{2}$. We check the condition (C1).

(C1) By the definition of $\theta_{\mathscr{IR}(Q)}$, we have that $\theta_{\mathscr{IR}(Q)}$ is an equivalence relation on $\mathscr{IR}(Q)$. Meanwhile, since $\theta$ is a congruence on $\mathbf{Q}$ and $(\mathscr{IR}(Q))^{\ast}=\mathscr{IR}(Q)$, we have $\theta_{\mathscr{IR}(Q)}^{\ast}=\theta_{\mathscr{IR}(Q)}$. For any $\langle a,b\rangle\in \theta_{\mathscr{IR}(Q)}$, then we have $\langle a\vee a,b\vee b\rangle\in \theta\cap \mathscr{R}(Q)^{2}= \theta_{\mathscr{R}(Q)}$ and then $(a\vee a)/\theta_{\mathscr{R}(Q)}=(b\vee b)/\theta_{\mathscr{R}(Q)}$, it turns out that $a\in cl(a\vee a)\subseteq cl((a\vee a)/\theta_{\mathscr{R}(Q)})$ and $b\in cl(b\vee b)\subseteq cl((b\vee b)/\theta_{\mathscr{R}(Q)})=cl((a\vee a)/\theta_{\mathscr{R}(Q)})$, so $\langle a,b \rangle \in (cl(a\vee a)/\theta_{\mathscr{R}(Q)})\times cl((a\vee a)/\theta_{\mathscr{R}(Q)})$. Denote $k=a\vee a$. Then we have $k\in \mathscr{R}(Q)$ and $\theta_{\mathscr{IR}(Q)}\subseteq \bigcup_{k\in \mathscr{R}({Q})}(cl(k/\theta_{\mathscr{R}(Q)})\times (cl(k/\theta_{\mathscr{R}(Q)}))$.

We denote $X=\{x/\theta_{\mathscr{R}(Q)} \,|\, x=y\vee y \hbox{\ where\ } y \in cl(x)\cap\mathscr{IR}(Q)\cap x/\theta\}$. Then $X\subseteq \mathscr{R}(Q)/\theta_{\mathscr{R}(Q)}$ and $X^{\ast}=X$.
Now, we define a mapping $f:X\rightarrow \mathscr{IR}(Q)/\theta_{\mathscr{IR}(Q)}$ by $f(x/\theta_{\mathscr{R}(Q)})=y/\theta_{\mathscr{IR}(Q)}$. If $X$ is empty, then $f$ is obviously injective. If $X$ is non-empty, now we check the condition (C2).

(C2) For any $a/\theta_{\mathscr{R}(Q)}\in X$, if $f(a/\theta_{\mathscr{R}(Q)})=b/\theta_{\mathscr{IR}(Q)}$ and $f(a/\theta_{\mathscr{R}(Q)})=c/\theta_{\mathscr{IR}(Q)}$ where $a=b\vee b=c\vee c$ and $b,c\in \mathscr{IR}(Q)\cap a/\theta$, then $\langle a,b\rangle\in \theta$ and $\langle a,c\rangle\in \theta$. Since $\theta$ is a congruence on $\mathbf{Q}$, we have $\langle b,c\rangle\in \theta$ and then $\langle b,c\rangle\in \theta\cap \mathscr{IR}(Q)^{2}=\theta_{\mathscr{IR}(Q)}$, so $c/\theta_{\mathscr{IR}(Q)}=b/\theta_{\mathscr{IR}(Q)}$. Thus $f$ is well-defined.
If $f(a_{1}/\theta_{\mathscr{R}(Q)})=f(a_{2}/\theta_{\mathscr{R}(Q)})$, then there exist $b_{1},b_{2}\in \mathscr{IR}(Q)$ satisfying $\langle a_{1},b_{1}\rangle,\langle a_{2},b_{2}\rangle\in \theta$ and  $b_{1}/\theta_{\mathscr{IR}(Q)}=b_{2}/\theta_{\mathscr{IR}(Q)}$, it turns out that $\langle b_{1},b_{2}\rangle\in \theta_{\mathscr{IR}(Q)}\subseteq\theta$. Since $\theta$ is a congruence on $\mathbf{Q}$, we have $\langle a_{1},a_{2}\rangle\in \theta\cap \mathscr{R}(Q)^{2}=\theta_{\mathscr{R}(Q)}$ and then $a_{1}/\theta_{\mathscr{R}(Q)}=a_{2}/\theta_{\mathscr{R}(Q)}$. Thus $\emph{f}$ is injective.
Finally, we check that $f$ preserves the operation $^{\ast}$.
If $f(a/\theta_{\mathscr{R}(Q)})=b/\theta_{\mathscr{IR}(Q)}$, then we have $a=b\vee b$ where $b\in cl(a)\cap\mathscr{IR}(Q)\cap a/\theta$. Since $\theta$ is a congruence on $\mathbf{Q}$ and $b^{\ast}\in cl(a^{\ast})$ by Lemma \ref{Lemma number}(2), we have $b^{\ast}\in a^{\ast}/\theta\cap cl(a^{\ast})$. Moreover, $b^{\ast}\in \mathscr{IR}(Q)$,  we have $a^{\ast}/\theta_{\mathscr{R}(Q)} \in X$ and $a^{\ast}=b^{\ast}\vee b^{\ast}$, so $f(a^{\ast}/\theta_{\mathscr{R}(Q)})=b^{\ast}/\theta_{\mathscr{IR}(Q)}=(b/\theta_{\mathscr{IR}(Q)})^{\ast}=f(a/\theta_{\mathscr{R}(Q)})^{\ast}$ by (C1).

Finally, we check the condition (C3).

(C3) Denote $\theta_{\mathscr{R}(Q)\times \mathscr{IR}(Q)}=\theta\setminus (\theta_{\mathscr{R}(Q)}\cup\theta_{\mathscr{IR}(Q)})$.
Then we have $\theta_{\mathscr{R}(Q)\times \mathscr{IR}(Q)}=\{\langle a,b\rangle\in \theta\,|\,(a\in \mathscr{R}(Q),b\in \mathscr{IR}(Q))$ or $(b\in \mathscr{R}(Q),a\in \mathscr{IR}(Q))\}$. Below we calculate that
\begin{align}
     \theta_{\mathscr{R}(Q)\times \mathscr{IR}(Q)}&=(\bigcup_{x/\theta_{\mathscr{R}(Q)}\in X}(x/\theta_{\mathscr{R}(Q)}\times f(x/\theta_{\mathscr{R}(Q)})))\cup (\bigcup_{x/\theta_{\mathscr{R}(Q)}\in X}(f(x/\theta_{\mathscr{R}(Q)}) \times x/\theta_{\mathscr{R}(Q)} )).\nonumber
\end{align}
Indeed, for $\langle a,b\rangle\in \theta_{\mathscr{R}(Q)\times \mathscr{IR}(Q)}$, we have $a\in \mathscr{R}(Q),b\in \mathscr{IR}(Q)$ or $b\in \mathscr{R}(Q),a\in \mathscr{IR}(Q)$. If the former holds, then there exist $x\in\mathscr{R}(Q)$ and $y\in\mathscr{IR}(Q)$ such that $\langle x,a\rangle\in \theta_{\mathscr{R}(Q)}$ and $\langle y,b\rangle\in \theta_{\mathscr{IR}(Q)}$. Since $\langle a,b\rangle, \langle x,a\rangle, \;\mathrm{and}\;\langle y,b\rangle \in \theta$, we have $\langle x,y\rangle\in \theta$, it implies that $y\in (x/\theta)\cap \mathscr{IR}(Q)$. Moreover, by (C1), we can get that $y\in cl((b\vee b)/\theta_{\mathscr{R}(Q)})=cl(a/\theta_{\mathscr{R}(Q)})=cl(x/\theta_{\mathscr{R}(Q)})$ and then $y\in cl(x)$. Hence $x/\theta_{\mathscr{R}(Q)}\in X$ and $f(x/\theta_{\mathscr{R}(Q)})=y/\theta_{\mathscr{IR}(Q)}$. Note that $a\in x/\theta_{\mathscr{R}(Q)}$ and $b\in f(x/\theta_{\mathscr{R}(Q)})$, so $\langle a,b\rangle\in x/\theta_{\mathscr{R}(Q)}\times f(x/\theta_{\mathscr{R}(Q)})$. If the latter holds, then we have $\langle a,b\rangle\in f(x/\theta_{\mathscr{R}(Q)}) \times x/\theta_{\mathscr{R}(Q)}$ similarly. Thus  \begin{align}
     \theta_{\mathscr{R}(Q)\times \mathscr{IR}(Q)}&=(\bigcup_{x/\theta_{\mathscr{R}(Q)}\in X}(x/\theta_{\mathscr{R}(Q)}\times f(x/\theta_{\mathscr{R}(Q)})))\cup (\bigcup_{x/\theta_{\mathscr{R}(Q)}\in X}(f(x/\theta_{\mathscr{R}(Q)}) \times x/\theta_{\mathscr{R}(Q)} )).\nonumber
\end{align}
\end{proof}

\begin{remark}
It should be noted that the set $X$ in Proposition \ref{theta123} is not necessarily $\mathscr{R}(Q)/\theta_{\mathscr{R}(Q)}$,
thereby these $\theta_{\mathscr{R}(Q)}$, $\theta_{\mathscr{IR}(Q)}$, and the mapping $f$  do not necessarily satisfy the conditions of Proposition  \ref{break}.

\end{remark}

At the end of this section, we demonstrate the methods discussed earlier with some examples.

\begin{example}\label{exa 4break}
Let $\mathbf{4}$ be the QB-algebra defined in Example \ref{Example 4}. Then $\mathscr{R}(4)=\{0,1\}$ and $\theta=\{\langle0,0\rangle,\langle0,1\rangle, \langle1,0\rangle,\langle1,1\rangle\}$ is a congruence on $\mathscr{R}(\mathbf{4})$. Following from Proposition \ref{xy con}, we have that $\theta_{a,b}=\theta \cup \Delta \cup \{\langle a,a\rangle, \langle b,b\rangle \}$ is a congruence on $\mathbf{4}$. In fact, we calculate that $\theta_{a,b}=\{\langle0,0\rangle,\langle0,1\rangle, \langle1,\\ 0\rangle,\langle1,1\rangle,\langle a,a\rangle, \langle b,b\rangle \}$ is indeed a congruence on $\mathbf{4}$. Moreover, following Proposition \ref{theta123}, we can decompose the congruence $\theta_{a,b}=\theta_{\mathscr{R}(4)}\cup\theta_{\mathscr{IR}(4)}\cup\theta_{\mathscr{R}(4)\times \mathscr{IR}(4)}$, where $\theta_{\mathscr{R}(4)}=\{\langle0,0\rangle,\langle0,1\rangle, \langle1,0\rangle,\\ \langle1,1\rangle\}$, $\theta_{\mathscr{IR}(4)}=\{\langle a,a\rangle, \langle b,b\rangle\}$, and
$\theta_{\mathscr{R}(4)\times \mathscr{IR}(4)}=\emptyset$.
\end{example}

\begin{example}
Let $\mathbf{F}_{3}$ be the flat QB-algebra defined in Example \ref{Example F3}. Then $c\neq d$, $c^{\ast}\neq c$, and $d^{\ast}\neq d$.
Following from Proposition $\ref{xy con2}$,  we have that $\theta_{c,d}=\Delta\cup \{\langle c,d\rangle, \langle d,c\rangle, \langle c^{\ast},d^{\ast}\rangle, \langle d^{\ast},c^{\ast}\rangle,\langle c,c^{\ast}\rangle,$ $\langle c^{\ast},c\rangle,\langle d,d^{\ast}\rangle,\langle d^{\ast},d\rangle,\langle c,d^{\ast}\rangle, \langle d^{\ast},c\rangle,\langle c^{\ast},d\rangle,\langle d,c^{\ast}\rangle \}=\{\langle0,0\rangle,\langle c,c\rangle, \langle d,d\rangle, \langle c,d\rangle,\langle d,c\rangle\}$ is a congruence on $\mathbf{F}_{3}$. Moreover, according to Proposition \ref{p5.3}, we can decompose the congruence $\theta_{c,d}=\Delta_{\mathscr{R}(\mathbf{F}_{3})}\cup \theta_{\mathscr{IR}(F_{3})}$, where $\Delta_{\mathscr{R}(\mathbf{F}_{3})}=\{\langle0,0\rangle\}$ and $\theta_{\mathscr{IR}(F_{3})}=\{\langle c,c\rangle, \langle d,d\rangle,\langle c,d\rangle, \langle d,c\rangle\}$.
\end{example}

\begin{remark}
In Example \ref{exa 4break}, we can see that $X_{4}=\emptyset$ and $\mathscr{IR}(4)/\theta_{\mathscr{IR}(4)}=\{\{a\},\{b\}\}$ following from Proposition \ref{theta123}. Thus the mapping $f$ is injective but not surjective.
\end{remark}

\begin{example}\label{E con}
Let $\mathbf{6}$ be the QB-algebra defined in Example \ref{Example E} and $\theta=\{\langle 0,0\rangle,\langle 0,1\rangle, \langle 1,0\rangle,\langle 1,1\rangle,\langle a,\\a\rangle,\langle e,e\rangle,\langle f,f\rangle,\langle b,b\rangle,\langle 0,a\rangle, \langle 0,e\rangle,\langle a,0\rangle,\langle e,0\rangle,\langle a,e\rangle,\langle e,a\rangle,\langle 1,f\rangle, \langle 1,b\rangle,\langle f,1\rangle,\langle b,1\rangle,\langle f,b\rangle,\langle b,f\rangle \}$ be a congruence on $\mathbf{6}$. Then following Proposition \ref{theta123}, we can decompose the congruence $\theta=\theta_{\mathscr{R}(6)}\cup\theta_{\mathscr{IR}(6)}\cup\theta_{\mathscr{R}(6)\times \mathscr{IR}(6)}$, where $\theta_{\mathscr{R}(6)}=\{\langle0,0\rangle,\langle0,1\rangle, \langle1,0\rangle,\langle1,1\rangle\}$, $\theta_{\mathscr{IR}(6)}=\{\langle a,a\rangle, \langle e,e\rangle,\\ \langle f,f\rangle,\langle b,b\rangle\}$, and
$\theta_{\mathscr{R}(6)\times \mathscr{IR}(6)}=\{\langle 0,a\rangle, \langle 0,e\rangle,\langle a,0\rangle, \langle e,0\rangle,\langle a,e\rangle,\langle e,a\rangle,\langle 1,f\rangle, \langle 1,b\rangle,\langle f,1\rangle,\langle b,1\rangle,\\ \langle f,b\rangle,\langle b,f\rangle\}$.
\end{example}

\begin{remark}
In Example \ref{E con}, we can see that $X_{6}=\{\{0\},\{1\}\}$ and $\mathscr{IR}(6)/\theta_{\mathscr{IR}(6)}=\{\{a,e\},\{f,b\}\}$ following from Proposition \ref{theta123}. Thus the mapping $f$ is bijective.
\end{remark}

\noindent\textbf{Acknowledgement}

This study was funded by Shandong Provincial Natural Science Foundation, China
(No. ZR2020MA041).

\raggedright
\end{document}